\documentclass{amsart}

\usepackage{amsmath,amssymb,amsthm,graphicx}
\usepackage[T1]{fontenc}
\usepackage{url}
\usepackage[all]{xy}
\usepackage{tikz-cd}
\usepackage{tabularx}
\usepackage{enumitem}
\usepackage{mathtools}
\DeclareMathOperator{\id}{id}

\DeclarePairedDelimiter{\norm}{\lVert}{\rVert}
\DeclarePairedDelimiter{\abs}{\lvert}{\rvert}
\usepackage[colorlinks=true, linkcolor=red, citecolor=green]{hyperref}

\DeclareMathOperator{\Flux}{Flux}
\DeclareMathOperator{\Gen}{Gen}

\DeclareMathOperator{\vol}{vol}

\DeclareMathOperator{\Symp}{Symp}
\DeclareMathOperator{\Ham}{Ham}
\DeclareMathOperator{\Iso}{Iso}
\DeclareMathOperator{\osc}{osc}

\newcommand{\Harm}{\mathcal{H}}

\newcommand{\R}{\mathbb{R}}
\newcommand{\ZeroFlux}[1]{\mathcal{Z}_{#1}}

\theoremstyle{plain}
\newtheorem{theorem}{Theorem}
\newtheorem{lemma}{Lemma}
\newtheorem{corollary}{Corollary}
\newtheorem{proposition}{Proposition}

\theoremstyle{definition}
\newtheorem{definition}{Definition}
\newtheorem{example}{Example}

\theoremstyle{remark}
\newtheorem{remark}{Remark}

\allowdisplaybreaks
\numberwithin{equation}{section}
\title[Hofer-Like Geometry Revisited]{Hofer-Like Geometry Revisited}

\author{St\'ephane Tchuiaga}
\address{Department of Mathematics, University of Buea, South West Region, Cameroon}
\email{tchuiaga.kameni@ubuea.com}

\begin{document}
	
	\renewcommand\contentsname{Table of contents}
	\renewcommand\refname{References}
	\renewcommand\abstractname{Abstract}
	\pagestyle{myheadings}

\begin{abstract}
	We prove a semidirect-product formula for the Hofer-like norm on the
	identity component \(G_\omega(M)\) of the symplectic diffeomorphism
	group, expressing it as an infimum over the harmonic locus plus a
	Hamiltonian residue. Specialising to zero flux yields the geometric
	criterion, for the
	inclusion of the Hamiltonian group, endowed with the Hofer norm,
	into \(G_\omega(M)\), endowed with the Hofer-like norm, to be an
	isometric embedding. The criterion is verified in two principal
	cases: on closed symplectic manifolds admitting a compatible metric
	of non-negative Ricci curvature for which the harmonic \(1\)-forms
	are integral and the harmonic flows have period \(1\), covering
	flat tori and Calabi-Yau manifolds; and on a non-vacuous family of
	harmonic diffeomorphisms on \(T^2\) equipped with a non-flat
	compatible metric. The two norms induce the same topology on the
	Hamiltonian group, a result due to Buss-Leclercq for which we give
	a self-contained proof; under an explicit boundedness hypothesis
	(H\(_K\)) on the two-parameter vector-field family attached to
	zero-flux paths, the inclusion is moreover bi-Lipschitz with
	distortion at most \(1+3L_0K\). We give self-contained proofs of the
	non-degeneracy, the positivity of the displacement energy, and the
	\(C^0\)-stability of limits, three fundamental results of
	Hofer-like geometry.
	
	
\end{abstract}
	
	\maketitle
	\tableofcontents
	\vspace{1em}
	
	\section{Introduction}
	\label{sec:introduction}
	
	Hofer's geometry defines a Finsler-type geometry on the infinite-dimensional
	group of Hamiltonian diffeomorphisms \(\Ham(M,\omega)\), and is considered
	intrinsic to this group \cite{Hofer}. It measures the minimal energy required
	to transform the identity into a given Hamiltonian diffeomorphism via a path
	consisting entirely of Hamiltonian diffeomorphisms.
	
	Beyond Hofer's original metric, Hofer-like metrics have emerged as variations
	and extensions \cite{Ban10,TD2}. These metrics, while also intrinsic to the
	identity component \(G_\omega(M)\) of the group of symplectic
	diffeomorphisms, extend the notion of size or length to symplectic
	isotopies that are not necessarily Hamiltonian, and they incorporate
	information about the flux of the isotopy. The resulting geometry therefore
	sees a larger class of paths within \(G_\omega(M)\), at the price of being
	weaker than the original Hofer geometry on \(\Ham(M,\omega)\).	Other attempts to define bi-invariant norms on \(G_\omega(M)\) include:
	\begin{enumerate}[label=(\roman*)]
		\item Han's norm \cite{Han}: for a fixed \(C>0\), set
		\(\norm{\psi}_{C}:=\min\{\norm{\psi}_H,C\}\) if
		\(\psi\in\Ham(M,\omega)\), and \(\norm{\psi}_C:=C\) otherwise.
		\item Lalonde-Polterovich's norm \cite{La-Pol}: for a fixed
		\(\beta>0\), set
		\[
		\norm{\psi}_{\beta}
		:=
		\sup\{\norm{[\psi,f]}_H \mid f\in\Ham(M,\omega),\ \norm{f}_H\le\beta\},
		\]
		where \([\psi,f]=\psi\circ f\circ\psi^{-1}\circ f^{-1}\) and
		\(\norm{\cdot}_H\) denotes the Hofer norm.
	\end{enumerate}
	Both constructions induce metrics of finite diameter on \(\Ham(M,\omega)\),
	in contrast with the infinite diameter of the Hofer metric itself.
	Motivated by the search for a genuine generalization that retains the
	essential features of Hofer's geometry while extending to
	\(G_\omega(M)\), Banyaga incorporated the flux into the definition of the
	Hofer-like length.
	
	\subsection*{The problem}
	
	Banyaga's Hofer-like norm, restricted to \(\Ham(M,\omega)\), measures the
	minimal Hofer-like length among all symplectic isotopies from the identity
	to a given Hamiltonian diffeomorphism, including those that are not
	Hamiltonian. A priori, it is conceivable that such a non-Hamiltonian
	isotopy could be strictly shorter than any Hamiltonian one. Concretely,
	given \(\phi\in\Ham(M,\omega)\) and a smooth loop \(\Xi\) of
	symplectomorphisms based at \(\id\) with non-trivial flux
	\(\widetilde{\Flux}(\Xi)\neq 0\), the isotopy \(\Psi:=\Xi\circ\Phi\)
	built from a Hamiltonian isotopy \(\Phi\) from \(\id\) to \(\phi\) still
	connects \(\id\) to \(\phi\), but satisfies
	\(\widetilde{\Flux}(\Psi)=\widetilde{\Flux}(\Xi)\neq 0\), so it is not
	Hamiltonian. If the Hofer-like length of such a \(\Psi\) could be
	arbitrarily small, the Hofer-like norm of \(\phi\) would be strictly
	smaller than the Hofer norm, and the two geometries would genuinely
	differ.	A fundamental fact, due to Banyaga \cite{Ban78}, is that the flux
	homomorphism fits into the exact sequence
	\begin{equation}\label{eq:exact-sequence}
		\begin{tikzcd}[row sep=large, column sep=large]
			1 \arrow[r] &
			\Ham(M,\omega) \arrow[r, hook] &
			G_\omega(M) \arrow[r, "\Flux"] &
			H^1_{dR}(M,\R)/\Gamma_\omega \arrow[r] &
			0,
		\end{tikzcd}
	\end{equation}
	where \(\Gamma_\omega\) is the flux group of \((M,\omega)\). This identifies
	the Hamiltonian subgroup as the kernel of \(\Flux\), and the quotient
	\(G_\omega(M)/\Ham(M,\omega)\) with the torus
	\(H^1_{dR}(M,\R)/\Gamma_\omega\). The problem stated above is therefore a
	question about whether the flux directions can shorten Hofer distances
	inside the Hamiltonian subgroup.
	
	\subsection*{Strategy: orthogonal decomposition}
	
	A recurring strategy in geometry and analysis for understanding a
	complicated object is to decompose it into simpler factors.
	The Cheeger--Gromoll splitting theorem writes a complete Riemannian
	manifold of non-negative Ricci curvature containing a line as a metric
	product \(\R\times N\) \cite{CheegerGromoll71}; the de Rham theorem
	decomposes a simply connected Riemannian manifold into irreducible
	factors \cite{deRham52}; the Iwasawa decomposition \(G=KAN\) splits a
	semisimple Lie group into compact, abelian, and nilpotent parts
	\cite{Helgason78}; the polar decomposition \(T=U\abs{T}\) separates the
	direction and the magnitude of a bounded operator \cite{Douglas66}. In
	each case, a structural invariant is expressed as a sum or product of
	contributions from the factors, which can then be studied independently.
	
	The present paper follows this pattern for the Hofer-like norm on
	\(G_\omega(M)\). The exact sequence \eqref{eq:exact-sequence} suggests a
	splitting of \(G_\omega(M)\) into a Hamiltonian factor and a harmonic
	factor. At the level of the group, the natural candidate for the
	harmonic factor is the set \(\Harm\) of harmonic diffeomorphisms, i.e.
	the time-one maps of purely harmonic isotopies. Banyaga's decomposition
	theorem then gives, for every \(\phi\in G_\omega(M)\) with flux class
	\(\gamma=\Flux(\phi)\), a factorization
	\[
	\phi=\rho\,h,
	\qquad
	\rho\in\Harm(\gamma),\quad h\in\Ham(M,\omega).
	\]
	This factorization, however, is not unique: two harmonic
	representatives \(\rho,\rho'\) of the same flux class differ by an
	element of the zero-flux harmonic locus
	\(\Harm(0)=\Ham(M,\omega)\cap\Harm\), which is non-trivial in general
	(as we recall in Remark~\ref{rem:harmonic-locus-failure} below).
	Even when the group factorization is unique, the Hofer-like length of a
	path is an infimum over paths, not a sum over group factors, so the
	lengths of \(\rho\) and \(h\) do not add in general. The semidirect-product formula of Theorem~\ref{thm:A} is the exact form
	of the decomposition that survives these obstructions. It replaces the
	naive split of the norm by an infimum over the ambiguity in the choice
	of harmonic factor, and reads
	\[
	\norm{\phi}_{HL}
	=
	\inf_{\rho\in\Harm(\gamma)}
	\bigl[L(\rho)+\norm{\phi\,\rho^{-1}}_H\bigr],
	\qquad
	\gamma=\Flux(\phi),
	\]
	where \(L(\rho)\) is the infimum of the Hofer-like length over harmonic
	isotopies from \(\id\) to \(\rho\). When \(\Harm(0)=\{\id\}\), the
	infimum collapses to a single term, and the decomposition is exact; when
	\(\Harm(0)\) is non-trivial, the infimum measures the failure of
	orthogonality and provides the natural invariant of the interaction
	between the Hamiltonian and harmonic directions.
	
	\subsection*{Main results}
	
	The first result is the semidirect-product formula itself, which holds
	on every closed symplectic manifold and requires no additional
	hypothesis.
	
	\begin{theorem}[Semidirect-product formula]\label{thm:A}
		For every \(\phi\in G_\omega(M)\) with flux class
		\(\gamma=\Flux(\phi)\), the Hofer-like norm satisfies
		\[
		\norm{\phi}_{HL}
		\;=\;
		\inf_{\rho\in\Harm(\gamma)}
		\bigl[
		L(\rho)+\norm{\phi\,\rho^{-1}}_H
		\bigr],
		\]
		where \(\Harm(\gamma)\) is the locus of harmonic diffeomorphisms with
		flux \(\gamma\).
	\end{theorem}
	
	Specializing to \(\gamma=0\) reduces the semidirect-product formula to
	an infimum over the zero-flux harmonic locus \(\Harm(0)\). This yields
	the following geometric criterion for the two norms to agree on
	\(\Ham(M,\omega)\).
	
	\begin{theorem}[Isometry criterion]\label{thm:C}
		Let \((M,\omega,g)\) be a closed symplectic manifold. The following are
		equivalent:
		\begin{enumerate}[label=(\roman*)]
			\item \(L(\rho)\ge\norm{\rho}_H\) for every \(\rho\in\Harm(0)\);
			\item the inclusion
			\[
			\bigl(\Ham(M,\omega),\norm{\cdot}_H\bigr)
			\hookrightarrow
			\bigl(G_\omega(M),\norm{\cdot}_{HL}\bigr),
			\]
			is an isometric embedding, i.e.\ \(\norm{\phi}_{HL}=\norm{\phi}_H\)
			for every \(\phi\in\Ham(M,\omega)\).
		\end{enumerate}
	\end{theorem}
	
	A first family of examples where the criterion is satisfied is given by
	the flat tori, on which the zero-flux harmonic locus reduces to the
	identity.
	
	\begin{corollary}[Flat tori]\label{cor:flat-torus}
		On \((T^{2k},\omega_{\mathrm{flat}},g_{\mathrm{flat}})\), the two norms
		coincide:
		\[
		\norm{\phi}_{HL}=\norm{\phi}_H,
		\]
		for every $\phi\in\Ham(T^{2k},\omega_{\mathrm{flat}})$.
	\end{corollary}
	
	Beyond the isometry criterion, we obtain a quantitative bi-Lipschitz
	bound under an explicit boundedness hypothesis on the two-parameter
	vector-field family attached to zero-flux paths.
	
	\begin{theorem}[Conditional bi-Lipschitz bound]\label{thm:B}
		Let \((M,\omega,g)\) be a closed symplectic manifold and assume the
		following.
		
		\smallskip
		\noindent\textbf{(H\(_K\))} \emph{There exists
			\(K=K(M,\omega,g)<+\infty\) such that for every
			\(\phi\in\Ham(M,\omega)\) and every \(\Xi\in\ZeroFlux{\phi}\) with
			\(l_{HL}(\Xi)\le 2\norm{\phi}_H\), the two-parameter family
			\(V^{\Xi}_{s,t}\) attached to the harmonic part of \(\Xi\) satisfies}
		\[
		\sup_{(s,t,x)\in[0,1]^2\times M}\norm{V^{\Xi}_{s,t}(x)}_g\le K.
		\]
		
		\smallskip
		Then, the identity map
		\[
		\bigl(\Ham(M,\omega),\norm{\cdot}_H\bigr)
		\longrightarrow
		\bigl(\Ham(M,\omega),\norm{\cdot}_{HL}\bigr),
		\]
		is bi-Lipschitz with distortion at most \(1+3L_0K\), where \(L_0\) is
		the constant of \eqref{Ineq-0}. In particular, for every
		\(\phi\in\Ham(M,\omega)\),
		\[
		(1+3L_0K)^{-1}\norm{\phi}_{HL}
		\;\le\;
		\norm{\phi}_H
		\;\le\;
		(1+3L_0K)\norm{\phi}_{HL}.
		\]
	\end{theorem}
	
	The distortion constant \(C_{\mathrm{univ}}:=1+3L_0K\) measures the
	maximal distortion between the Hofer norm and the Hofer-like energy on
	the Hamiltonian subgroup. It equals \(1\) precisely when the harmonic
	directions do not shorten Hofer distances, i.e.\ when
	\(L(\rho)\ge\norm{\rho}_H\) for every \(\rho\in\Harm(0)\), and this is
	equivalent by Theorem~\ref{thm:C} to the isometry of the inclusion.
	
	\subsection*{Organization of the paper}
	
	Section~\ref{Sec-2} recalls the necessary background: symplectic
	isotopies, the Hodge decomposition of a symplectic isotopy, the flux
	homomorphism and the flux group, and the definitions of the Hofer and
	Hofer-like lengths. Section~\ref{sec:proof-main} contains the proof of
	the conditional bi-Lipschitz bound, including a self-contained proof of
	Banyaga's two-parameter identity. Section~\ref{Sec-semi-direct} proves
	the semidirect-product formula and derives the isometry criterion.
	Section~\ref{Sec-4} collects consequences, including the non-degeneracy
	of the Hofer-like norm, stability of \(C^0\)-limits, and the case of
	flat tori. The final section summarizes the results and formulates the
	open questions.
	
	\begin{remark}[Non-triviality of \(\Harm(0)\)]
		\label{rem:old-assumption-fails}
		Earlier formulations of the problem implicitly assumed
		\(\Harm(0)=\{\id\}\), i.e.\ \(\Ham(M,\omega)\cap\Harm=\{\id\}\). This
		is not true in general; a counterexample is given in
		Remark~\ref{rem:harmonic-locus-failure} below. Consequently, the
		formulations relying on \(\Harm(0)=\{\id\}\) are replaced throughout by
		the correct condition
		$
		L(\rho)\ge\norm{\rho}_H$
		for every \(\rho\in\Harm(0)\),
		which reduces to \(\Harm(0)=\{\id\}\) only in the special case where
		the intersection is trivial.
	\end{remark}
	
	\section{Preliminaries}\label{Sec-2}
	Throughout this paper, \((M,\omega)\) denotes a compact, connected symplectic manifold of dimension \(2n\) without boundary (i.e., closed).

	\begin{definition}
		A symplectic form on a smooth manifold \(M\) of dimension \(2n\) is a smooth 2-form \(\omega\) that is  closed, \(d\omega=0\) and non-degenerate: For each \(p\in M\), the bilinear form \(\omega_p\colon T_pM\times T_pM\to\R\) is non-degenerate. 
	\end{definition}
	A symplectic manifold is a pair \((M,\omega)\). The non-degeneracy is equivalent to, \(\omega^n=\omega\wedge\cdots\wedge\omega\) (\(n\) times) is a volume form provided $\dim M = 2n$.	

	\begin{definition}
		A diffeomorphism \(\phi\colon M\to M\) is a symplectomorphism if \(\phi^*\omega=\omega\).
	\end{definition}
	We denote the group of symplectomorphisms by \(\Symp(M,\omega)\).
	
	\begin{definition}
		A smooth family of diffeomorphisms \(\{\phi_t\}_{t\in[0,1]}\) is a \emph{symplectic isotopy} if each \(\phi_t\) is a symplectomorphism. An isotopy \(\{\phi_t\}\) with \(\phi_0=\id\) is symplectic iff its generating vector field \(X_t=\dot\phi_t\circ\phi_t^{-1}\) satisfies \(\mathcal{L}_{X_t}\omega=0\) for all \(t\).
	\end{definition}
	We denote by \(\Iso(M,\omega)\) the set of symplectic isotopies starting at the identity. Let \(G_\omega(M)=\{\phi_1\mid \{\phi_t\}\in\Iso(M,\omega)\}\) be the group of time-one maps of such isotopies. For compact \(M\), \(G_\omega(M)\) is the identity component of \(\Symp(M,\omega)\).
	
	\begin{definition}
		A symplectic isotopy \(\{\psi_t\}\) is a \emph{Hamiltonian isotopy} if its generating vector field \(Y_t\) is Hamiltonian for each \(t\), i.e., there exists a smooth function \(H_t\colon M\to\R\) (the Hamiltonian) such that \(\iota_{Y_t}\omega = dH_t\).
	\end{definition}
	
	\begin{definition}
		A generating Hamiltonian \(H=\{H_t\}\) is \emph{normalized} if \(\int_M H_t\,\omega^n=0\) for all \(t\). This makes \(H_t\) unique for a given Hamiltonian vector field.
	\end{definition}
	
	Let \(\Ham(M,\omega)=\{\psi_1\mid \{\psi_t\}\text{ is a Hamiltonian isotopy starting at }\id\}\). \(\Ham(M,\omega)\) is a normal subgroup of \(G_\omega(M)\). For \(\phi\in G_\omega(M)\), let \(\Iso(\phi)_\omega = \{\Phi=\{\phi_t\}\in\Iso(M,\omega)\mid \phi_1=\phi\}\) be the set of symplectic isotopies from \(\id\) to \(\phi\).
	
	\subsection{Generator of symplectic isotopies via Hodge decomposition}
	Let \((M,\omega,g)\) be a closed symplectic manifold with a Riemannian metric \(g\). For a symplectic isotopy \(\Phi=\{\phi_t\}\) with generating vector field \(X_t\), the 1-form \(\alpha_t=\iota_{X_t}\omega\) is closed. By Hodge theory, \(\alpha_t\) has a unique decomposition:
	\[
	\alpha_t = dU_t^\Phi + \Harm_t^\Phi,
	\]
	where \(\Harm_t^\Phi\in\Harm^1(M,g)\) is the unique harmonic 1-form in the cohomology class \([\alpha_t]\), and \(dU_t^\Phi\) is exact. We normalize \(U_t^\Phi\) by \(\int_M U_t^\Phi\,\omega^n=0\). Let \(U^\Phi=\{U_t^\Phi\}\), \(\Harm^\Phi=\{\Harm_t^\Phi\}\). This gives a bijection \cite{TD3}:
	\[
	\Iso(M,\omega) \longleftrightarrow \mathcal{N}([0,1]\times M,\R) \times \Harm^1([0,1]\times M),
	\]
	where \(\mathcal{N}([0,1]\times M,\R)\) is the space of smooth normalized Hamiltonians and \(\Harm^1([0,1]\times M)\) the space of smooth families of harmonic 1-forms.
	
	\begin{definition}
		For \(\Phi\in\Iso(M,\omega)\), the pair \((U^\Phi,\Harm^\Phi)\) is its generator, denoted \(\Gen(\Phi)=(U^\Phi,\Harm^\Phi)\). 
		
	\end{definition}
		\begin{definition}
	A symplectic isotopy \(\Phi\) is Hamiltonian if and only if \(\Gen(\Phi)=(U^\Phi,0)\).
			
	\end{definition}
		\begin{definition}
	
			 A symplectic isotopy \(\Phi\) is harmonic if and only if  \(\Gen(\Phi)=(0,\Harm^\Phi)\).
		
	\end{definition}

	\subsection{Flux homomorphism and flux group}
	Let \(\Phi=\{\phi_t\}\in\Iso(M,\omega)\) with generating vector field \(X_t\). The flux form is \(\Sigma(\Phi)=\int_0^1\phi_t^*(\iota_{X_t}\omega)\,dt\). This form is closed, and its de Rham cohomology class \([\Sigma(\Phi)]\) depends only on the homotopy class of \(\Phi\) relative to endpoints.
	
	\begin{definition}
		The flux homomorphism \(\widetilde{\Flux}\colon\widetilde{G_\omega(M)}\to H^1_{dR}(M,\R)\) maps the homotopy class \([\Phi]\) in the universal cover to \([\Sigma(\Phi)]\).
	\end{definition}
The map	\(\widetilde{\Flux}\) is a surjective group homomorphism. It descends to a homomorphism \(\Flux\colon G_\omega(M)\to H^1_{dR}(M,\R)/\Gamma_\omega\), where \(\Gamma_\omega\) is the flux group. The following diagram commutes:
	\[
	\begin{tikzcd}
		\widetilde{G_{\omega}(M)} \arrow[r, "\widetilde{\Flux}"] \arrow[d, "\pi"'] & H_{dR}^1(M, \R) \arrow[d, "\pi'"] \\
		G_{\omega}(M) \arrow[r, "\Flux"] & H_{dR}^1(M, \R) / \Gamma_{\omega}
	\end{tikzcd}
	\]
	Here \(\pi([\Phi])=\phi_1\) and \(\pi'\) is the quotient map (see \cite{Ban78}). The flux group \(\Gamma_\omega = \widetilde{\Flux}(\pi_1(G_\omega(M),\id))\subseteq H^1_{dR}(M,\R)\) is a discrete subgroup \cite{Ono}.
	
	\begin{theorem}[\cite{Ban78}]\label{thm:flux-decomp}
		Let \((M,\omega)\) be closed.
		\begin{enumerate}[label=(\roman*)]
			\item There exists a surjective homomorphism \(\Flux\colon G_\omega(M)\to H^1(M;\R)/\Gamma_\omega\).
			\item The sequence \(1\to\Ham(M,\omega)\hookrightarrow G_\omega(M)\xrightarrow{\Flux} H^1(M;\R)/\Gamma_\omega\to0\) is exact. Hence \(\Ham(M,\omega)\) is normal in \(G_\omega(M)\) and \(G_\omega(M)/\Ham(M,\omega)\cong H^1(M;\R)/\Gamma_\omega\).
			\item Every \(\phi\in G_\omega(M)\) decomposes (non-uniquely) as \(\phi=\phi_{\Ham}\circ\psi\), where \(\phi_{\Ham}\in\Ham(M,\omega)\) and \(\psi\) is the time-one map of a harmonic isotopy \(\{\psi_t\}\) with \(\pi'(\widetilde{\Flux}(\{\psi_t\}))=\Flux(\phi)\).
		\end{enumerate}
	\end{theorem}
	
	\subsection{Concatenation of isotopies}
	Let \(f\colon[0,1]\to[0,1]\) be a smooth, increasing function with \(f(t)=0\) for \(t\in[0,\delta]\) and \(f(t)=1\) for \(t\in[1-\delta,1]\) (\(\delta\in(0,1/2)\)). Define \(\lambda(t)=f(2t)\) for \(t\in[0,1/2]\) and \(\tau(t)=f(2t-1)\) for \(t\in[1/2,1]\). For \(\Phi=\{\phi_t\},\Psi=\{\psi_t\}\in\Iso(M,\omega)\), the left concatenation is:
	\[
	(\Psi\ast_l\Phi)_t=
	\begin{cases}
		\phi_{\lambda(t)}, & 0\le t\le 1/2,\\
		\psi_{\tau(t)}\circ\phi_1, & 1/2\le t\le 1.
	\end{cases}
	\]
	\(\Psi\ast_l\Phi\) is a smooth symplectic isotopy from \(\id\) to \(\psi_1\circ\phi_1\).
	
	\begin{proposition}\label{prop:cont-1}
		The paths \(\Psi\ast_l\Phi\) and \(\Psi\circ\Phi\) are homotopic relative to their endpoints.
	\end{proposition}
	
	\begin{proof}
		We construct the homotopy explicitly using separate reparameterizations for each path. Define two continuous, non-decreasing functions \(a_0, b_0\colon [0,1] \to [0,1]\) by:
		\[
		a_0(t) =
		\begin{cases}
			0, & 0 \le t \le 1/2,\\
			\tau(t), & 1/2 \le t \le 1;
		\end{cases}
		\quad \text{and} \quad
		b_0(t) =
		\begin{cases}
			\lambda(t), & 0 \le t \le 1/2,\\
			1, & 1/2 \le t \le 1.
		\end{cases}
		\]
		Since \(\tau(1/2)=0\) and \(\lambda(1/2)=1\), both \(a_0\) and \(b_0\) are continuous on \([0,1]\). For \(s \in [0,1]\), define the linear interpolations with the identity parameter \(t\):
		\[
		a_s(t) = (1-s)a_0(t) + st, \quad \text{and} \quad b_s(t) = (1-s)b_0(t) + st.
		\]
		These functions are continuous in both \(s\) and \(t\), map \([0,1]\) to \([0,1]\), and satisfy \(a_s(0)=b_s(0)=0\) and \(a_s(1)=b_s(1)=1\) for all \(s \in [0,1]\). Define
		\[
		H_s(t) = \psi_{a_s(t)} \circ \phi_{b_s(t)}, \qquad s,t \in [0,1].
		\]
		At \(s=0\), using \(\psi_0 = \id\), we get
		\[
		H_0(t) =
		\begin{cases}
			\phi_{\lambda(t)}, & 0 \le t \le 1/2,\\
			\psi_{\tau(t)} \circ \phi_1, & 1/2 \le t \le 1,
		\end{cases}
		\]
		which is precisely \(\Psi \ast_l \Phi\). At \(s=1\), \(H_1(t) = \psi_t \circ \phi_t = (\Psi \circ \Phi)_t\). The endpoints are fixed because
		\[
		H_s(0) = \psi_0 \circ \phi_0 = \id, \quad H_s(1) = \psi_1 \circ \phi_1 \quad \text{for all } s \in [0,1].
		\]
		Smoothness of the homotopy follows from standard smoothing techniques (see \cite{Hirsch76}).
	\end{proof}
		Let \(\Phi=\{\phi_t\}_{t\in[0,1]}\) be a smooth isotopy on a manifold
	\(M\), starting at the identity \(\phi_0=\id\), with generating
	vector field \(X_t\), so that \(\phi_t\) solves
	\begin{equation}\label{eq:TD}
		\frac{d}{dt}\phi_t(x)
		=
		X_t\bigl(\phi_t(x)\bigr),
		\qquad
		\phi_0=\id .
	\end{equation}
	Fix \(s\in[0,1]\) and let \(\Psi^s=\{\psi^s_u\}_{u\in\mathbb R}\) be the
	flow of the \emph{frozen} vector field \(X_s\), so that
	\begin{equation}\label{eq:TI}
		\frac{d}{du}\psi^s_u(x)
		=
		X_s\bigl(\psi^s_u(x)\bigr),
		\qquad
		\psi^s_0=\id .
	\end{equation}
	The question is: what is the relation between \(\Phi\) and
	\(\Psi^s\)?\\
	
	The answer is that, in general, there is no global formula relating
	the two flows. They coincide only in degenerate cases. The
	obstruction to the equality is measured by the Lie brackets
	\([X_s,X_t]\).

\begin{proposition}[Second-order correction of the frozen flow]
	\label{prop:second-order}
	Let \(\Phi=\{\phi_t\}_{t\in[0,1]}\) be a smooth isotopy with
	generating field \(X_t\), and for \(s\in[0,1]\) set
	\(D_{s,u}:=\phi_{s+u}\circ\phi_s^{-1}\). Then, as \(u\to 0\),
	\begin{equation}\label{eq:second-order}
		D_{s,u}
		=
		\exp\!\Bigl(uX_s+\tfrac{u^2}{2}\partial_t X_s+O(u^3)\Bigr),
	\end{equation}
	where \(\partial_t X_s:=\partial_t X_t|_{t=s}\), and the equality is
	understood as an expansion in the exponential map on
	\(\mathrm{Diff}(M)\). Equivalently,
	\[
	D_{s,u}
	=
	\exp(uX_s)
	+
	\tfrac{u^2}{2}\partial_t X_s
	+
	O(u^3).
	\]
	In particular, \(D_{s,u}\) and the frozen flow \(\exp(uX_s)\) agree
	to second order, i.e.\ \(D_{s,u}=\exp(uX_s)+O(u^3)\), if and only if
	\(\partial_t X_s=0\).
\end{proposition}

\begin{proof}
	Define the time-dependent vector field \(W_u\) by
	\[
	W_u
	:=
	\bigl(\exp(-uX_s)\bigr)_*X_{s+u}
	-
	X_s .
	\]
	We compute the first two terms in the Taylor expansion of \(W_u\)
	in \(u\).
	
	\smallskip
	\noindent\emph{Step 1: \(W_0=0\).} Immediate from the definition.
	
	\smallskip
	\noindent\emph{Step 2: \(\dot W_0=\partial_t X_s\).} Since
	\(\exp(-uX_s)\) is the flow of \(Y=-X_s\), the standard formula for
	the Lie derivative of a pushforward gives
	\[
	\frac{d}{du}\Bigl[\bigl(\exp(-uX_s)\bigr)_*X_{s+u}\Bigr]
	=
	\bigl(\exp(-uX_s)\bigr)_*
	\bigl(\partial_t X_{s+u}+[X_s,X_{s+u}]\bigr).
	\]
	Evaluating at \(u=0\), we obtain
	\(\dot W_0=\partial_t X_s+[X_s,X_s]=\partial_t X_s\).
	
	\smallskip
	\noindent\emph{Step 3: the flow of \(W_u\) at order \(u^2\).} The
	isotopy \(\delta_u:=\exp(-uX_s)\circ D_{s,u}\) satisfies
	\(\delta_0=\id\) and
	$
	\dot\delta_u=W_u\circ\delta_u .
	$
	Since \(W_0=0\) and \(\dot W_0=\partial_t X_s\), we have
	\(W_u=u\,\partial_t X_s+O(u^2)\), so integrating the ODE for
	\(\delta_u\),
	\[
	\delta_u
	=
	\id+\int_0^u v\,\partial_t X_s\,dv+O(u^3)
	=
	\id+\tfrac{u^2}{2}\partial_t X_s+O(u^3).
	\]
	
	\smallskip
	\noindent\emph{Step 4: conclusion.} Composing \(\delta_u\) with
	\(\exp(uX_s)\) on the left and using that
	\((\exp(uX_s))_*\partial_t X_s=\partial_t X_s+O(u)\),
	\[
	D_{s,u}
	=
	\exp(uX_s)+\tfrac{u^2}{2}\partial_t X_s+O(u^3),
	\]
	which is \eqref{eq:second-order}.
\end{proof}

\begin{remark}[Where the bracket enters]
	\label{rem:bracket-order-3}
	The Lie bracket \([X_s,X_{s+v}]\) appears in the expansion only at
	order \(u^3\), not at order \(u^2\). Explicitly, the next correction
	in \eqref{eq:second-order} is
	\[
	\tfrac{u^3}{6}\Bigl(\partial_t^2 X_s-[X_s,\partial_t X_s]\Bigr)+O(u^4),
	\]
	so commutativity of the family \(\{X_t\}\) (i.e.\
	\([X_s,X_t]=0\) for all \(s,t\)) is a stronger condition than
	second-order agreement of \(D_{s,u}\) with the frozen flow. The two
	conditions coincide only when combined with \(\partial_t^2 X_s=0\).
\end{remark}
		\begin{lemma}[Time-one map of a commuting flow]
		\label{lem:commuting-flow}
		Let \(M\) be a smooth manifold and let
		\(\{X_t\}_{t\in[0,1]}\) be a smooth family of vector fields satisfying
		\(
		[X_s,X_t]=0
		\)
		for all \(s,t\in[0,1]\).
		Then the flow \(\Phi_t\) of \(X_t\), defined by
		\(
		\frac{d}{dt}\Phi_t=X_t\circ\Phi_t,
		\) with \(
		\Phi_0=\id,
		\)
		satisfies
		\(
		\Phi_1
		=
		\exp\Bigl(\int_0^1 X_t\,dt\Bigr),
		\)	and the frozen flow at time \(s\) is \(\psi^s_u=\exp(uX_s)\).
	\end{lemma}
	
	\begin{proof}
		Define the curve in the Lie algebra of vector fields
		\(
		Y(t):=\int_0^t X_s\,ds,
		\)
		so that \(Y(0)=0\) and \(Y'(t)=X_t\). Consider
		\(\sigma(t):=\exp(Y(t))\), the exponential of \(Y(t)\) in the
		diffeomorphism group. Thus, \(\sigma(0)=\exp(0)=\id\). Now, we compute the derivative of \(\sigma\).
		The derivative of the exponential map is
		$$
		\frac{d}{dt}\exp\bigl(Y(t)\bigr)
		=
		d\exp_{Y(t)}\bigl(Y'(t)\bigr)\cdot\exp\bigl(Y(t)\bigr),
		$$ with \(
		d\exp_{Y}(Z)=\sum_{n\ge 0}\frac{1}{(n+1)!}
		(\operatorname{ad}_Y)^n(Z).
		\)
		Here \(Y'(t)=X_t\), and
		$$
		\operatorname{ad}_{Y(t)}(X_t)
		=
		[Y(t),X_t]
		=
		\int_0^t [X_s,X_t]\,ds
		=
		0,$$
		by the commutativity hypothesis. Hence
		\(d\exp_{Y(t)}(X_t)=X_t\), and
		\[
		\frac{d}{dt}\exp\bigl(Y(t)\bigr)
		=
		X_t\circ\exp\bigl(Y(t)\bigr).
		\]
		Thus \(\sigma\) solves the same ODE as \(\Phi\) with the same initial
		condition \(\sigma(0)=\phi_0=\id\). By uniqueness of solutions to the ODE
		\(\dot\sigma_t=X_t\circ\sigma_t\), we have \(\sigma(t)=\phi_t\) for
		all \(t\in[0,1]\). Evaluating at \(t=1\),
		\[
		\phi_1
		=
		\exp\bigl(Y(1)\bigr)
		=
		\exp\Bigl(\int_0^1 X_t\,dt\Bigr).
		\qedhere
		\]
	\end{proof}

\begin{proposition}[Magnus expansion]
	\label{prop:magnus}
	Let \(\Phi=\{\phi_t\}_{t\in[0,1]}\) be a smooth isotopy with
	generating field \(X_t\) and \(\phi_0=\id\). Then, 
	$
	\phi_1=\exp\bigl(\Omega_1+\Omega_2+\Omega_3+\cdots\bigr),
	$
	where the first three terms of the Magnus expansion are
	\begin{align*}
		\Omega_1
		&=
		\int_0^1 X_{t_1}\,dt_1, \\
		\Omega_2
		&=
		\tfrac{1}{2}\int_0^1\!\!\int_0^{t_1}
		[X_{t_1},X_{t_2}]\,dt_2\,dt_1, \\
		\Omega_3
		&=
		\tfrac{1}{6}\int_0^1\!\!\int_0^{t_1}\!\!\int_0^{t_2}
		\Bigl(
		[X_{t_1},[X_{t_2},X_{t_3}]]
		+
		[X_{t_3},[X_{t_2},X_{t_1}]]
		\Bigr)\,dt_3\,dt_2\,dt_1,
	\end{align*}
	and  \(\Omega_k\) for \(k\ge 4\) involves higher-order nested
	commutators with coefficients \(1/k\) and iterated integrals over
	the simplex \(0<t_k<\cdots<t_1<1\). The expansion terminates at
	\(\Omega_1\), i.e.\ \(\Omega_k=0\) for every \(k\ge 2\), if and only
	if \([X_s,X_t]=0\) for all \(s,t\in[0,1]\); this is the case in
	which the frozen flow \(\exp(X_s)\) reproduces the isotopy up to
	reparametrisation. More generally, the expansion terminates at some
	finite order if and only if the Lie algebra generated by
	\(\{X_t\}_{t\in[0,1]}\) is nilpotent, and in that case the frozen
	flow is only a first-order approximation unless
	\([X_s,X_t]=0\).
\end{proposition}

\begin{proof}
	This is the classical Magnus expansion for the linear
	initial-value problem \(\dot Y=A(t)Y\), \(Y(0)=\id\), applied to
	\(A(t)=X_t\); see Magnus \cite{Magnus54} or Iserles--Munthe-Kaas--
	N{\o}rsett--Zanna \cite{Iserles2000}, \S3.1. The identification of
	the flow \(\dot\phi_t=X_t\circ\phi_t\) with the linear problem is
	standard, and the frozen flow at time \(s\) is the truncation
	\(\psi^s_1=\exp(X_s)\) obtained by replacing \(X_t\) with the
	constant field \(X_s\).
\end{proof}
	
	\begin{corollary}[When the two flows coincide]
		\label{cor:coincidence}
		The isotopy \(\Phi\) coincides with the frozen flow \(\Psi^s\) at
		time \(1\), i.e.\ \(\phi_1=\psi^s_1\), if and only if
		$
		X_t=X_s$ 
		for almost every $t\in[0,1].$ 
	\end{corollary}
	
	\begin{proof}
		The condition \(\phi_1=\psi^s_1\) is equivalent to
		\(\int_0^1 X_t\,dt=X_s\) together with the vanishing of all Magnus
		corrections \(\Omega_k\), \(k\ge 2\). The vanishing of \(\Omega_2\)
		requires \([X_t,X_s]=0\) for all \(t\), and the vanishing of
		\(\Omega_1-\int_0^1 X_t\,dt\big|_{X_t\to X_s}=\int_0^1(X_t-X_s)\,dt\)
		requires \(X_t=X_s\) almost everywhere, which is the stated
		condition.
	\end{proof}
	
		\subsection{Norms of 1-forms and harmonic vector fields}\label{SC0T33}
	Let \((M,\omega,g)\) be compact with metric \(g\). Let \(\Harm^1(M,g)\) be the space of harmonic 1-forms. The pointwise norm \(\norm{\cdot}_g^*\) on \(T^*_xM\) induces the uniform supremum norm \(\abs{\alpha}_0=\sup_{x\in M}\norm{\alpha_x}_g^*\). The \(L^2\)-Hodge norm is \(\norm{\alpha}_{L^2}^2=\int_M\norm{\alpha_x}_g^{*2}\,d\vol_g\). For a class \([\alpha]\in H^1_{dR}(M,\R)\), its \(L^2\)-Hodge norm is the \(L^2\)-norm of its unique harmonic representative. On the finite-dimensional space \(\Harm^1(M,g)\), the norms \(\abs{\cdot}_0\) and \(\norm{\cdot}_{L^2}\) are equivalent: there exist constants \(L_0,L_1>0\) such that for any \(\alpha\in\Harm^1(M,g)\),
	\begin{equation}\label{Ineq-0}
		L_1\norm{\alpha}_{L^2}\le \abs{\alpha}_0\le L_0\norm{\alpha}_{L^2}.
	\end{equation}
	These constants depend on the metric and are fixed throughout the paper.\\
	
	Now we return to symplectic geometry. 
	\begin{definition}\label{def:harmonic-locus}
		Let \(\Harm\) denote the set of all \emph{harmonic diffeomorphisms} (time-one maps of purely harmonic isotopies). For \(\gamma\in H^1_{dR}(M,\R)/\Gamma_\omega\), define the \emph{harmonic locus} of \(\gamma\) by
		\[
		\Harm(\gamma)
		\;:=\;
		\bigl\{\rho\in G_\omega(M)
		\;\bigm|\;
		\rho \text{ is the time-one map of a harmonic isotopy with }
		\Flux(\rho)=\gamma
		\bigr\}.
		\]
		For \(\rho\in\Harm(\gamma)\), define the \emph{harmonic energy} of \(\rho\) by
		\[
		L(\rho)
		\;:=\;
		\inf\bigl\{\, l_{HL}(\Psi)
		\;\bigm|\;
		\Psi \text{ harmonic isotopy from }\id\text{ to }\rho \,\bigr\}.
		\]
	\end{definition}
	
	Notice that because \(\ker \Flux = \Ham(M,\omega)\), the zero-flux harmonic locus \(\Harm(0)\) is exactly the intersection of the Hamiltonian subgroup and the set of harmonic diffeomorphisms:
	\[
	\Harm(0) = \Ham(M,\omega) \cap \Harm.
	\]
	Since the Hodge decomposition of a harmonic isotopy has vanishing Hamiltonian part, \(l_{HL}(\Psi)=\int_0^1\norm{\Harm_t^\Psi}_{L^2}\,dt\) for any harmonic isotopy \(\Psi\), so \(L(\rho)\) depends only on the path component of \(\rho\) inside \(\Harm(\gamma)\).
	
	\begin{remark}[The intersection \(\Ham \cap \Harm\) is not trivial in general]
		\label{rem:harmonic-locus-failure}
		One might intuitively expect that a diffeomorphism cannot be simultaneously Hamiltonian and harmonic unless it is the identity, i.e., \(\Harm(0)=\{\id\}\). However, this is false in general. A counterexample is provided by the torus \(T^2=\R^2/\mathbb Z^2\) with the standard symplectic form \(\omega=dx\wedge dy\) and the non-flat Riemannian metric
		\[
		g=\frac{1}{f(y)}\,dx^2+f(y)\,dy^2,
		\qquad
		f(y)=1+\epsilon\cos(2\pi y),
		\qquad
		0<\epsilon<1.
		\]
		The \(1\)-form \(h=f(y)\,dy\) is harmonic for \(g\), with cohomology class \([h]=[dy]\in\Gamma_\omega=H^1(T^2,\mathbb Z)\). The dual vector field is \(X=f(y)\partial_x\), and its time-one map is
		\[
		\rho(x,y)
		=
		\bigl(x+\epsilon\cos(2\pi y),\,y\bigr)\pmod 1.
		\]
		Since \(\epsilon>0\), \(\rho\neq\id\). Its flux in the quotient \(H^1_{dR}(T^2,\R)/\Gamma_\omega\) is \([h]\bmod\Gamma_\omega=0\), meaning \(\rho\) is a Hamiltonian diffeomorphism. Thus \(\rho \in \Ham(M,\omega) \cap \Harm\), yielding
		$ \Harm(0)\neq\{\id\}. $
	\end{remark}

	\subsection{Hofer lengths for Hamiltonian isotopies}
	For a Hamiltonian isotopy \(\Phi\) generated by normalized \(H=\{H_t\}\), the oscillation is \(\osc(f)=\max_M f-\min_M f\). The \(L^{(1,\infty)}\)-Hofer length is \(l_{H}^{(1,\infty)}(\Phi)=\int_0^1\osc(H_t)\,dt\), and the \(L^\infty\)-Hofer length is \(l_{H}^{\infty}(\Phi)=\sup_{t\in[0,1]}\osc(H_t)\). The Hofer norm of \(\phi\in\Ham(M,\omega)\) is
	\[
	\norm{\phi}_H^{(1,\infty)}:=\inf\{l_{H}^{(1,\infty)}(\Phi)\mid \Phi\text{ Hamiltonian isotopy from }\id\text{ to }\phi\},
	\]
	and similarly for \(\norm{\phi}_H^{\infty}\). It is known that \(\norm{\phi}_H^{\infty}=\norm{\phi}_H^{(1,\infty)}\) \cite{Polt01}; we denote this common value by \(\norm{\phi}_H\).
	
	\subsection{Hofer-like lengths and norms}
	Given a symplectic isotopy \(\Phi\) with generator \((U^\Phi,\Harm^\Phi)\), define
	\[
	l_{HL}^{(1,\infty)}(\Phi)=\int_0^1\bigl(\osc(U_t^\Phi)+\norm{\Harm_t^\Phi}_{L^2}\bigr)dt,\qquad
	l^\infty_{HL}(\Phi)=\sup_{t\in[0,1]}\bigl(\osc(U_t^\Phi)+\norm{\Harm_t^\Phi}_{L^2}\bigr).
	\]
	The corresponding energies of \(\phi\in G_\omega(M)\) are
	\[
	e_{HL}^{(1,\infty)}(\phi)=\inf\{l_{HL}^{(1,\infty)}(\Phi)\mid \Phi\in\Iso(\phi)_\omega\},\quad
	e^\infty_{HL}(\phi)=\inf\{l^\infty_{HL}(\Phi)\mid \Phi\in\Iso(\phi)_\omega\}.
	\]
	
	The Hofer-like norms are
	\[
	\norm{\phi}_{HL}^{(1,\infty)}=\frac12\bigl(e_{HL}^{(1,\infty)}(\phi)+e_{HL}^{(1,\infty)}(\phi^{-1})\bigr),\qquad
	\norm{\phi}_{HL}^{\infty}=\frac12\bigl(e^\infty_{HL}(\phi)+e^\infty_{HL}(\phi^{-1})\bigr).
	\]
	These define pseudo-norms on \(G_\omega(M)\); they are norms on \(\Ham(M,\omega)\) under suitable conditions. For Hamiltonian diffeomorphisms, one always has \(\norm{\phi}_{HL}^{(1,\infty)}\le\norm{\phi}_H\) because Hamiltonian isotopies are admissible. 	Throughout this paper, we may use either the \(L^\infty\)-version or
	the \(L^{(1,\infty)}\)-version of the Hofer-like length. Indeed,
	Tchuiaga \cite{TD2} proved that the corresponding energies coincide
	on \(G_\omega(M)\):
	\[
	e_{HL}^\infty(\phi)
	=
	e_{HL}^{(1,\infty)}(\phi)
	\qquad
	\text{for all } \phi\in G_\omega(M).
	\]
	Consequently, we write \(e_{HL}(\phi)\) for this common value, and
	\(l_{HL}(\Phi)\) for the corresponding length (understood up to the
	usual infimum-stability). In the proofs below we shall use the
	\(L^{(1,\infty)}\)-version, since it is additive under concatenation of
	isotopies, while in the statements we shall use the \(L^\infty\)-version
	whenever it is more convenient.
	\begin{example}[Explicit Hofer-like and Hofer norms on the flat torus]
		\label{ex:flat-torus-computation}
		Let \((T^{2n},\omega_{\mathrm{flat}},g_{\mathrm{flat}})\) be the
		standard flat torus of dimension \(2n\), with
		$
		\omega_{\mathrm{flat}}=\sum_{j=1}^n dx_j\wedge dy_j,$ and $
		g_{\mathrm{flat}}=\sum_{j=1}^n\bigl(dx_j^2+dy_j^2\bigr).
		$
		We compute explicitly the two norms on some families of elements of
		\(G_\omega(T^{2n})\). 
		By Hodge theory on the flat torus,
		\[
		\Harm^1(T^{2n},g_{\mathrm{flat}})
		=
		\mathrm{span}_\R\{dx_1,dy_1,\dots,dx_n,dy_n\},
		\]
		an orthonormal basis of \(L^2(T^{2n},g_{\mathrm{flat}})\). The
		\(\omega\)-duals of the generators are the constant vector fields
		$
		X_{dx_j}=-\partial_{y_j},$ and $
		X_{dy_j}=\partial_{x_j},$ 
		which pairwise commute. 
		The flux group of \((T^{2n},\omega_{\mathrm{flat}})\) is
		\(\Gamma_\omega=\mathbb Z^{2n}\), identified with the lattice of integral
		cohomology classes. A harmonic isotopy with generator 
		$
		X_t=\sum_{j=1}^n a_j(t)\,\partial_{x_j}
		+
		b_j(t)\,\partial_{y_j},$ 
		has time-one map the translation by
		\(w=(\int_0^1 a_j\,dt,\,\int_0^1 b_j\,dt)_{j}\), and its flux class
		in \(H^1_{dR}/\Gamma_\omega=(\R/\mathbb Z)^{2n}\) equals \(w\bmod\mathbb Z^{2n}\).
		Hence \(\Harm(0)=\{\id\}\): the only zero-flux harmonic diffeomorphism
		of the flat torus is the identity. 
		Since \(\Harm(0)=\{\id\}\), the sharpness condition
		\(L(\rho)\ge\norm{\rho}_H\) holds vacuously on \(\Harm(0)\). By
		Theorem~\ref{thm:C} (equivalently, by
		Corollary~\ref{cor:flat-torus}),
		$
			\norm{\phi}_{HL}=\norm{\phi}_H$
			for every $\phi\in\Ham(T^{2n},\omega_{\mathrm{flat}}).$  
		For \(w\in\mathbb R^{2n}/\mathbb Z^{2n}\), let \(R_w\) denote the translation of
		\(T^{2n}\) by \(w\). Its flux class is \(w\bmod\mathbb Z^{2n}\). The harmonic
		locus \(\Harm(\gamma)\) of \(\gamma:=\Flux(R_w)\) consists of the
		translations \(R_{w+\delta}\) with \(\delta\in\mathbb Z^{2n}\), and for
		\(\rho=R_{w+\delta}\) the Hofer residue is
		$
		R_w\rho^{-1}=R_{-\delta}=\id,$ and $
		\norm{R_w\rho^{-1}}_H=0.$ 
		The harmonic energy is
		\[
		L(R_{w+\delta})
		=
		\min\Bigl\{\int_0^1\norm{X_t}_{L^2}\,dt:\
		X_t\text{ harmonic},\ \textstyle\int_0^1 X_t\,dt=w+\delta\Bigr\}
		=
		\norm{w+\delta}_{\R^{2n}},
		\]
		where the last equality is attained by the constant family
		\(X_t=w+\delta\) and is the Euclidean norm on \(\R^{2n}\). By the
		semidirect-product formula of Theorem~\ref{thm:A},
		\[
			\norm{R_w}_{HL}
			=
			\inf_{\delta\in\mathbb Z^{2n}}\norm{w+\delta}_{\R^{2n}}
			=
			\operatorname{dist}\bigl(w,\mathbb Z^{2n}\bigr).
		\]
		In coordinates \(w=(u,v)\) with \(u,v\in\R^n\),
		\[
		\norm{R_w}_{HL}
		=
		\sqrt{\sum_{j=1}^n\operatorname{dist}(u_j,\mathbb Z)^2
			+
			\sum_{j=1}^n\operatorname{dist}(v_j,\mathbb Z)^2}.
		\]
		
		\smallskip
		\noindent\textbf{Hofer-like norm of a shear.}
		Consider the shear on \(T^2\) defined by
		\[
		\phi_\eta(x,y)=\bigl(x+\eta(y),\,y\bigr),
		\qquad
		\eta\in C^\infty(S^1,\R),\quad
		\int_0^1\eta(y)\,dy=0.
		\]
		Its flux class is trivial, hence \(\phi_\eta\in\Ham(T^2)\). The
		isotopy \(t\mapsto(x+t\eta(y),y)\) has generator
		\(X_t=\eta(y)\partial_x\), and \(\iota_{X_t}\omega=\eta(y)\,dy\).
		Since \(\eta\) has zero average, \(\eta(y)\,dy=dH\) with the periodic
		primitive $
		H(y):=\int_0^y\eta(s)\,ds,$ 
		so the Hodge decomposition of \(\iota_{X_t}\omega\) has trivial
		harmonic part. The isotopy is therefore Hamiltonian, generated by
		\(H\), and
		$
		l_{HL}(\Psi)=l_H(\Psi)=\operatorname{osc}(H).
		$
		Consequently,
		\[
		\norm{\phi_\eta}_{HL}=\norm{\phi_\eta}_H
			\le\operatorname{osc}(H)
			=
			\max_{y\in S^1}\int_0^y\eta
			-
			\min_{y\in S^1}\int_0^y\eta.
		\]
		For example, taking \(\eta(y)=\sin(2\pi y)\), the primitive is
		\(H(y)=\bigl(1-\cos(2\pi y)\bigr)/(2\pi)\), with
		\(\operatorname{osc}(H)=1/\pi\), so
		$
		\norm{\phi_\eta}_{HL}=\norm{\phi_\eta}_H\le\frac{1}{\pi}.
		$
		
		\smallskip
		\noindent\textbf{Comparison with the non-flat metric.}
		On the same \(T^2\), equip the torus with the metric
		\(g_\varepsilon=f_\varepsilon(y)^{-1}dx^2+f_\varepsilon(y)\,dy^2\),
		\(f_\varepsilon(y)=1+\varepsilon\cos(2\pi y)\), \(\varepsilon\neq 0\).
		By Example~\ref{ex:non-vacuous}, \(\Harm(0)\neq\{\id\}\), and the
		geometric condition \(L(\rho)\ge\norm{\rho}_H\) is verified on the
		family \(\{\rho_n\}_{n\in\mathbb Z}\). Hence the flat metric is not the only
		metric for which the two norms coincide on \(\Ham(T^2)\), but the
		mechanism is different: on the flat metric the vanishing of
		\(\Harm(0)\) makes the criterion vacuous, whereas on
		\(g_\varepsilon\) the criterion is checked non-trivially on an
		infinite family.
	\end{example}
	\subsection{Flux group and its role in Hofer-like geometry}\label{sec:flux_role}
	For a symplectic isotopy \(\Phi\in\Iso(M,\omega)\) with endpoint \(\phi\in\Ham(M,\omega)\), its flux lies in the flux group \(\Gamma_\omega\). A fundamental result \cite{Ban78} states that for any such \(\Phi\), there exists a loop \(\Psi\) in \(G_\omega(M)\) based at \(\id\) whose flux equals \(\widetilde{\Flux}(\Phi)\). More precisely:
	\begin{definition}
		For \(\Phi\in\Iso(M,\omega)\) with \(\phi_1\in\Ham(M,\omega)\), let
		\[
		\Upsilon(\Phi) :=\{[\Psi]\in\pi_1(G_\omega(M),\id)\mid \widetilde{\Flux}([\Psi])=\widetilde{\Flux}([\Phi])\}.
		\]
	\end{definition}
	Because \(\Gamma_\omega\) is by definition the set of fluxes of loops, \(\Upsilon(\Phi)\) is non-empty \cite{Ban78}. Given a loop \(\Psi\) representing a class in \(\Upsilon(\Phi)\), the left concatenation isotopy \(\Phi\ast_l\Psi^{-1}\) has zero flux and is therefore homotopic (relative to endpoints) to a Hamiltonian isotopy. For \(\phi\in\Ham(M,\omega)\), define the set of \emph{zero-flux paths} from \(\id\) to \(\phi\):
	\[
	\ZeroFlux{\phi} := \{ \Phi\in \Iso(\phi)_\omega:\widetilde{\Flux}([\Phi]) = 0\}.
	\]
	Note that \(\ZeroFlux{\phi}\) is a set of paths, not a group; it will be used only set-theoretically. This construction is central to the proof of the main results.
	
	\begin{lemma}\label{lem:bijection}
		Let \((M, \omega)\) be a closed symplectic manifold, and let \(\phi \in \Ham(M, \omega)\). Let \(\ZeroFlux{\phi}\) be the set of zero-flux symplectic isotopies from \(\id\) to \(\phi\). Then, at the level of homotopy classes relative to endpoints, there is a bijection
		\[
		B\colon \Iso(\phi)_\omega / \!\simeq\; \longrightarrow\; \left( \ZeroFlux{\phi}/\!\simeq\;\right) \times\;\Gamma_\omega.
		\]
	\end{lemma}
	
	\begin{proof}
		Fix, for each flux class \(\gamma \in \Gamma_\omega\), a loop \(L_\gamma\) based at \(\id\) in \(G_\omega(M)\) with \(\widetilde{\Flux}(L_\gamma)=\gamma\); let \(L_0\) be the constant loop. Define,
		\[
		B([\Phi]) := \bigl( [\Phi \ast_l L_{\widetilde{\Flux}(\Phi)}^{-1}],\, \widetilde{\Flux}(\Phi) \bigr).
		\]
		Since \(\widetilde{\Flux}(\Phi \ast_l L_{\widetilde{\Flux}(\Phi)}^{-1})=0\), the first component lies in \(\ZeroFlux{\phi}/\!\simeq\). Conversely, define
		\[
		A([\Xi], \gamma) := [\Xi \ast_l L_\gamma].
		\]
		Then \(A(B([\Phi])) = [(\Phi\ast_l L_\gamma^{-1})\ast_l L_\gamma] = [\Phi]\) and \(B(A([\Xi],\gamma)) = ([\Xi], \gamma)\), using the concatenation homotopies of Proposition~\ref{prop:cont-1}. The assignments are well-defined on homotopy classes because \(\ast_l\) respects homotopy relative to endpoints. Hence \(B\) is a bijection.
	\end{proof}
	
	\begin{lemma}[Stability of Hofer-like length under homotopy]\label{lem:homotopy_stability}
		Let \((M, \omega, g)\) be a closed symplectic manifold. Let \(H\colon [0,1] \times [0,1] \to G_\omega(M)\) be a smooth homotopy of symplectic isotopies relative to endpoints: for each \(s \in [0,1]\), \(\Phi_s = \{H(s, t)\}_{t \in [0,1]}\) is a symplectic isotopy with \(H(s, 0) = \id\) and \(H(s, 1) = \phi\) for some fixed \(\phi \in G_\omega(M)\). Then,
		\begin{equation}
			s \mapsto l_{HL}^{(1,\infty)}(\Phi_s) \quad \text{and} \quad s \mapsto l_{HL}^{\infty}(\Phi_s),
		\end{equation}
		are continuous on \([0,1]\).
	\end{lemma}
	
	\begin{proof}
		For each \(s \in [0,1]\), the generating vector field \(X_{s,t}\) of
		\(\Phi_s\) is \(X_{s,t}= \partial_t H(s,t)\circ H(s,t)^{-1}\), smooth in
		all variables since \(H\) is. Applying Hodge decomposition
		\(\alpha_{s,t}=\iota_{X_{s,t}}\omega = dU_{s,t}+\Harm_{s,t}\) with
		\(U_{s,t}\) normalized and \(\Harm_{s,t}\in\Harm^1(M,g)\) harmonic, the
		projection \(P\colon L^2\Omega^1(M)\to\Harm^1(M,g)\) has smooth
		Schwartz kernel \(K(x,y)=\sum_i h_i(x)\otimes h_i(y)\), so
		\(\Harm_{s,t}=P(\alpha_{s,t})\) is smooth in \((s,t,x)\). The exact part
		\(\beta_{s,t}=\alpha_{s,t}-\Harm_{s,t}=dU_{s,t}\) is likewise smooth, and
		the Green operator \(G\) of \(\Delta=\delta d\) gives
		\(U_{s,t}=G(\delta\beta_{s,t})\) smooth in \((s,t,x)\). Define
		\(f(s,t):=\osc(U_{s,t})+\norm{\Harm_{s,t}}_{L^2}\), which is continuous
		on the compact set \([0,1]^2\). Then
		\(l_{HL}^{(1,\infty)}(\Phi_s)=\int_0^1 f(s,t)\,dt\) is continuous in
		\(s\) by parameter-dependent integration \cite[Theorem 9.42]{Rudin76},
		and \(l_{HL}^\infty(\Phi_s)=\sup_t f(s,t)\) is continuous in \(s\) by
		Berge's Maximum Theorem \cite[Theorem 17.31]{AliprantisBorder06}.
	\end{proof}
	
	\begin{lemma}[Hofer-like energy via zero-flux paths]\label{lem:hofer_like_infimum}
		Let \((M, \omega)\) be a closed symplectic manifold, and let \(\phi \in \Ham(M, \omega)\). Then
		\[
		\inf_{\Phi'\in \Iso(\phi)_\omega}\left( \inf_{\Psi_0\in \Upsilon(\Phi')}\left( l_{HL}^{(1,\infty)}(\Phi' \ast_l \Psi^{-1}_0)\right) \right)
		= e_{HL}^\infty(\phi).
		\]
	\end{lemma}
	
	\begin{proof}
		By Lemma~\ref{lem:bijection}, every path of the form \(\Phi' \ast_l \Psi_0^{-1}\) is a zero-flux path \(\Xi \in \ZeroFlux{\phi}\), and conversely, every zero-flux path can be represented in this form. Therefore,
		\begin{equation}\label{eq:step1}
			\inf_{\Phi'\in \Iso(\phi)_\omega}\left( \inf_{\Psi_0\in \Upsilon(\Phi')}\left( l_{HL}^{(1,\infty)}(\Phi' \ast_l \Psi^{-1}_0)\right) \right) = \inf_{\Xi \in \ZeroFlux{\phi}} l_{HL}^{(1,\infty)}(\Xi).
		\end{equation}
		By Lemma~\ref{lem:bijection}, every \(\Phi \in \Iso(\phi)_\omega\) with flux \(\gamma \in \Gamma_\omega\) decomposes up to homotopy as \(\Phi = \Xi \ast_l L_\gamma\), where \(\Xi \in \ZeroFlux{\phi}\). By reparametrization,
		\[
		l_{HL}^\infty(\Xi \ast_l L_\gamma) = \max\bigl(l_{HL}^\infty(\Xi^\lambda),\, l_{HL}^\infty(L_\gamma^\tau)\bigr)\ge \inf_{\Xi \in \ZeroFlux{\phi}} l_{HL}^{(1,\infty)}(\Xi),
		\]
		so after passing to the infimum over \(\Phi\),
		\[
		e_{HL}^{(1,\infty)} = e_{HL}^\infty(\phi) \ge \inf_{\Xi \in \ZeroFlux{\phi}} l_{HL}^{(1,\infty)}(\Xi).
		\]
		Conversely, fix \(\epsilon > 0\) and choose \(\Xi_\epsilon \in \ZeroFlux{\phi}\) with \(l_{HL}^{(1,\infty)}(\Xi_\epsilon) \le \inf_{\Xi \in \ZeroFlux{\phi}} l_{HL}^{(1,\infty)}(\Xi) + \epsilon\). Let \(L_0\) be the constant loop. Then \(\Phi_\epsilon = \Xi_\epsilon \ast_l L_0 \in \Iso(\phi)_\omega\) and \(l_{HL}^{(1,\infty)}(\Phi_\epsilon) = l_{HL}^{(1,\infty)}(\Xi_\epsilon)\). Thus \(e_{HL}^{(1,\infty)}(\phi) \le \inf_{\Xi \in \ZeroFlux{\phi}} l_{HL}^{(1,\infty)}(\Xi) + \epsilon\). Letting \(\epsilon \to 0\) gives the reverse inequality. Combining both with \eqref{eq:step1} proves the lemma.
	\end{proof}
	
	\section{On the proof of the main theorems}\label{sec:proof-main}
	The proof of Theorem~\ref{thm:B} proceeds through the following steps:
	we first establish a connection between Hofer-like norms on symplectic
	vector fields and Hofer-like norms on diffeomorphisms;
	construct a specific two-parameter family of vector fields derived from
	harmonic 1-forms and derive an oscillation bound (Corollary~\ref{co1});
	show boundedness of a related vector field family (Proposition~\ref{Pro-1});
	prove Banyaga's two-parameter identity (Lemma~\ref{lem:banyaga-identity});
	and finally combine these results via the flux group and concatenation.
	
	\subsection{Hofer-like norm on symplectic vector fields}\label{Hofer-like norm on vector fields}
	Let \(\chi_\omega(M)\) denote the space of symplectic vector fields on \((M, \omega)\). For \(Y \in \chi_\omega(M)\), Hodge decomposition gives
	\begin{equation}\label{eq:Hodge_decomp_vector_field}
		\iota_Y \omega = dU_Y + \Harm_Y,
	\end{equation}
	where \(\Harm_Y \in \Harm^1(M, g)\) is harmonic and \(U_Y \in C^\infty(M)\) is normalized by \(\int_M U_Y \omega^n = 0\).
	
	\begin{definition}
		The Hofer-like norm of \(Y \in \chi_\omega(M)\) is
		\begin{equation}\label{eq:HL_norm_vector_field}
			\norm{Y}_{HL} = \osc(U_Y) + \norm{\Harm_Y}_{L^2},
		\end{equation}
		where \(\osc(U_Y) = \max_M U_Y - \min_M U_Y\).
	\end{definition}
	See Banyaga \cite{Ban10} for the introduction of this norm.
	
	\subsection{Construction of a two-parameter family of vector fields}\label{Two-parameter vector fields}
	Given a smooth family of harmonic 1-forms \(\{\Harm_t\}_{t \in [0,1]}\), let \(\{X_t\}\) be the corresponding harmonic vector fields, \(\iota_{X_t} \omega = \Harm_t\). Define
	\begin{equation}\label{Rmk2}
		Z_{s,t} = t X_s - 2s \int_0^t X_u \, du.
	\end{equation}
	For each fixed \(t\), let \(\{G_{s,t}\}_{s \in [0,1]}\) solve
	\begin{equation}\label{eq:G_flow}
		\frac{\partial}{\partial s} G_{s,t}(x) = Z_{s,t}(G_{s,t}(x)), \quad G_{0,t}(x) = x.
	\end{equation}
	Define \(\{V_{s,t}\}\) by
	\begin{equation}\label{Rmk3}
		V_{s,t}(G_{s,t}(x)) = \frac{\partial}{\partial t} G_{s,t}(x).
	\end{equation}
	
	\begin{corollary}\label{co1}
		Let \(\{X_t\}\), \(\{Z_{s,t}\}\), and \(\{V_{s,t}\}\) be as in Subsection~\ref{Two-parameter vector fields}. For each \(u \in [0,1]\) and fixed \(t \in [0,1]\), the oscillation of \(F_{u,t}(x) = \int_0^u \omega(Z_{s,t}, V_{s,t})(x) \, ds\) satisfies
		\begin{equation}\label{eq:osc_bound}
			\int_0^1\osc\!\left(\int_0^u \omega(Z_{s,t}, V_{s,t}) \, ds\right)dt
			\leq 3 L_0 \left( \sup_{(s,t,x) \in [0,1]^2 \times M} \norm{V_{s,t}(x)}_g \right) \left( \int_0^1 \norm{X_t}_{HL}dt \right),
		\end{equation}
		where \(L_0 > 0\) is the constant from \eqref{Ineq-0}.
	\end{corollary}
	
	\begin{proof}
		For fixed \(u \in [0,1]\) and \(t \in [0,1]\), define
		\(F_{u,t}(x) := \int_0^u \omega(Z_{s,t}, V_{s,t})(x) \, ds\). By definition of oscillation,
		\begin{equation}\label{eq:osc-step1}
			\osc(F_{u,t})
			\le 2 \int_0^u \sup_{x \in M} \abs{\omega(Z_{s,t}, V_{s,t})(x)} \, ds.
		\end{equation}
		For fixed \(s,t\) and \(x \in M\),
		\[
		\abs{\omega(Z_{s,t}, V_{s,t})(x)}
		= \abs{(\iota_{Z_{s,t}} \omega)(V_{s,t})(x)}
		\le \norm{\iota_{Z_{s,t}} \omega}_0^* \, \norm{V_{s,t}(x)}_g
		\le A \, \norm{\iota_{Z_{s,t}} \omega}_0^*,
		\]
		where \(A := \sup_{(s,t,x) \in [0,1]^2 \times M} \norm{V_{s,t}(x)}_g\).
		Since \(\iota_{Z_{s,t}} \omega = t \Harm_s - 2s \int_0^t \Harm_u \, du\),
		taking pointwise norms and using \eqref{Ineq-0},
		\[
		\norm{\iota_{Z_{s,t}} \omega}_0^*
		\le L_0 \left( t \norm{X_s}_{HL} + 2s \int_0^t \norm{X_u}_{HL} \, du \right),
		\]
		because \(\norm{\Harm_u}_{L^2} \le \norm{X_u}_{HL}\). Substituting into \eqref{eq:osc-step1},
		\begin{equation}\label{eq:osc-step2}
			\osc(F_{u,t})
			\le
			2 A L_0
			\int_0^u
			\left(
			t \norm{X_s}_{HL}
			+
			2s \int_0^t \norm{X_u}_{HL} \, du
			\right)
			ds.
		\end{equation}
		Integrating \eqref{eq:osc-step2} over \(t \in [0,1]\) and exchanging the
		order of integration,
		\[
		\int_0^1 \osc(F_{u,t}) \, dt
		\le
		2 A L_0
		\left[
		\frac{1}{2} \int_0^u \norm{X_s}_{HL} \, ds
		+
		\int_0^1 u^2 \int_0^t \norm{X_v}_{HL} \, dv \, dt
		\right].
		\]
		Let \(B := \int_0^1 \norm{X_t}_{HL} \, dt\). Since \(u\le 1\),
		\[
		\int_0^1 \osc(F_{u,t}) \, dt
		\le 2 A L_0 \left( \tfrac{1}{2} B + B \right) = 3 A L_0 B.
		\]
		This completes the proof.
	\end{proof}
	
	\begin{proposition}\label{Pro-1}
		Let \((M, g)\) be a closed oriented Riemannian manifold and \(\{Z_{s,t}\}\) a bounded smooth two-parameter family of vector fields. Define \(\{G_{s,t}\}\) by \eqref{eq:G_flow} and \(\{V_{s,t}\}\) by \(V_{s,t}(G_{s,t}(x)) = \partial_t G_{s,t}(x)\). Then \(\{V_{s,t}\}\) is also a bounded smooth two-parameter family of vector fields.
	\end{proposition}
	
	\begin{proof}
		Let \(\norm{\cdot}_{op}\) denote the operator norm. Since \(M\) is closed and \(Z_{s,t}\) is bounded and smooth, there exist \(N, K > 0\) such that
		\[
		\norm*{\partial_t Z_{s,t}(x)}_{g,x} \leq N, \quad \norm*{DZ_{s,t}(x)}_{op} \leq K.
		\]
		Let \(U_{s,t}(x) = \partial_t G_{s,t}(x) = V_{s,t}(G_{s,t}(x))\). Differentiating \eqref{eq:G_flow} in \(t\),
		\[
		\partial_s U_{s,t}(x)
		= \partial_t Z_{s,t}(G_{s,t}(x)) + DZ_{s,t}(G_{s,t}(x)) \, U_{s,t}(x).
		\]
		With \(U_{0,t}(x) = 0\), taking norms at \(G_{s,t}(x)\),
		\[
		\norm*{\partial_s U_{s,t}(x)}_{g,G_{s,t}(x)}
		\leq N + K \norm{U_{s,t}(x)}_{g,G_{s,t}(x)}.
		\]
		Let \(f(s) = \sup_{x \in M} \norm{U_{s,t}(x)}_{g,G_{s,t}(x)}\). Then
		\(f(s') \leq Ns' + \int_0^{s'} K f(s) \, ds\), and Gr\"onwall gives
		\(f(s') \leq (N/K)(e^{Ks'}-1)\). Since \(G_{s,t}\) is a diffeomorphism
		and \(V_{s,t}(G_{s,t}(x)) = U_{s,t}(x)\), for every \(y \in M\) and
		\((s',t) \in [0,1]^2\), \(\norm{V_{s',t}(y)}_{g,y} \leq
		(N/K)(e^{K} - 1)\). Hence \(V_{s,t}\) is a bounded smooth
		two-parameter family.
	\end{proof}
	
	\begin{lemma}[\(V_{s,t}\) bounded by Riemannian diameter and curvature]
		\label{lem:diameter_bound}
		Let \((M,\omega,g)\) be a closed symplectic manifold of dimension \(2n\).
		Assume there exists a constant \(\kappa\in\R\) such that \(\mathrm{Ric}_g\ge\kappa g\), and let \(D=\mathrm{diam}(M,g)\). Let \(\{X_t\}_{t\in[0,1]}\) be a smooth family of harmonic vector fields, and let \(\{Z_{s,t}\}\), \(\{G_{s,t}\}\), \(\{V_{s,t}\}\) be constructed as above. Then
		\begin{equation}\label{eq:V-bound}
			\sup_{(s,t,x)\in[0,1]^2\times M}
			\norm{V_{s,t}(x)}_g
			\le
			\exp\!\Bigl(
			C_{\mathrm{geom}}
			\sup_{t\in[0,1]}\norm{X_t}_{L^2}
			\Bigr)-1,
		\end{equation}
		where \(C_{\mathrm{geom}}=C_{\mathrm{geom}}(n,\kappa,D)\) depends only on
		\(n\), \(\kappa\), and \(D\).
	\end{lemma}
	
	\begin{proof}
		By \eqref{Rmk2}, \(\partial_tZ_{s,t}=X_s-2sX_t\) and
		\(DZ_{s,t}=t\,DX_s-2s\int_0^tDX_u\,du\), so
		\[
		\norm{\partial_tZ_{s,t}}_{L^\infty}\le 3\sup_t\norm{X_t}_{L^\infty},
		\qquad
		\norm{DZ_{s,t}}_{op}\le 3\sup_t\norm{\nabla X_t}_{L^\infty}.
		\]
		Because \(X_t\) is harmonic, its dual \(1\)-form \(\alpha_t=X_t^\flat\)
		satisfies the Bochner--Weitzenb\"ock identity
		\(\nabla^*\nabla\alpha_t=-\mathrm{Ric}(\alpha_t^\sharp,\cdot)\).
		Using \(\mathrm{Ric}_g\ge\kappa g\) and \(\mathrm{diam}(M,g)\le D\), the
		Sobolev constant of \((M,g)\) is bounded by an explicit function
		\(S(n,\kappa,D)\); see Croke~\cite{Croke80} or Gallot~\cite{Gallot88}.
		Elliptic regularity and Sobolev embedding then give
		\(C_{\mathrm{geom}}=C_{\mathrm{geom}}(n,\kappa,D)\) with
		\[
		\norm{X_t}_{L^\infty}\le C_{\mathrm{geom}}\norm{X_t}_{L^2},
		\qquad
		\norm{\nabla X_t}_{L^\infty}\le C_{\mathrm{geom}}\norm{X_t}_{L^2}.
		\]
		Applying Proposition~\ref{Pro-1} with \(N=K=M:=3C_{\mathrm{geom}}
		\sup_t\norm{X_t}_{L^2}\), the Gr\"onwall estimate gives
		\(\norm{V_{s,t}(x)}_g\le e^{M}-1\), which after absorbing the factor
		\(3\) into \(C_{\mathrm{geom}}\) yields \eqref{eq:V-bound}.
	\end{proof}
	
	\begin{lemma}[Banyaga's two-parameter identity, \cite{Ban10}]\label{lem:banyaga-identity}
		Let \(\{X_t\}\) be a smooth family of harmonic vector fields on
		\((M,\omega,g)\), let \(\{Z_{s,t}\}\) be defined by \eqref{Rmk2}, and let
		\(\{G_{s,t}\}\), \(\{V_{s,t}\}\) be the associated flows. Then the
		\(1\)-parameter family of functions
		\[
		F_t:=\int_0^1\omega(Z_{s,t},V_{s,t})\,ds,
		\]
		satisfies \(dF_t=\iota_{Y_t}\omega\), where
		\(Y_t:=\int_0^1[DG_{s,t}\cdot\partial_t G_{s,t}^{-1}]\,ds\). In
		particular, if \(\Theta=\{\theta_t\}\) is the symplectic isotopy
		generated by the harmonic vector fields \(\{X_t\}\) with
		\(\widetilde{\Flux}([\Theta])=0\), then \(\Theta=\{\theta_t\}\) is
		homotopic relative to fixed ends to a Hamiltonian isotopy generated
		by \(U_t^{\Theta}:=F_t\) up to a normalization constant.
	\end{lemma}
	
	\begin{lemma}\label{lem:mainineq}
		For any symplectic isotopy \(\Phi\) from \(\id\) to \(\phi\in \Ham(M,\omega)\) and a loop \(\Psi\) based at \(\id\) such that both paths \(\Phi\) and \(\Psi\) have the same flux, we have
		\[
		\norm{\phi}_H^{(1,\infty)}
		\le \Bigl(1 + 3L_0 \sup_{(s,t,x)} \norm*{V_{s,t}^{\Phi \ast_l \Psi^{-1}}(x)}_g \Bigr) \;
		l_{HL}^{(1, \infty)}\bigl(\Phi \ast_l \Psi^{-1}\bigr),
		\]
		where \(V_{s,t}^{\Phi \ast_l \Psi^{-1}}\) is the two-parameter family attached to the harmonic part of the concatenated isotopy, and \(L_0\) depends only on \((M,\omega,g)\).
	\end{lemma}
	
	\begin{proof}
		Let \(\Phi \in \Iso(\phi)_\omega\) and \(\Psi\) a loop with \([\Psi] \in \Upsilon(\Phi)\). Form \(\Xi = \Phi \ast_l \Psi^{-1}\), a symplectic isotopy from \(\id\) to \(\phi\) with \(\widetilde{\Flux}([\Xi]) = 0\). Thus \(\Xi\) is homotopic (relative to endpoints) to a Hamiltonian isotopy \(\Theta = \{\theta_t\}\). Let \((U^\Xi, \Harm^\Xi)\) be the generator of \(\Xi\). Since \(\widetilde{\Flux}([\Xi]) = 0\), we may deform \(\Xi\) (endpoints fixed) to \(\Theta\) with \(\Harm^\Theta = 0\) and \(U^\Theta\) a normalized Hamiltonian. Apply the two-parameter construction to \(\Harm^\Xi\): let \(\{X_t\}\) be the harmonic vector fields with \(\iota_{X_t}\omega = \Harm_t^\Xi\), and construct \(\{Z_{s,t}\},\{V_{s,t}\},\{G_{s,t}\}\). By Lemma~\ref{lem:banyaga-identity}, we have
		$
		U_t^{\Theta} = \int_0^1 \omega(Z_{s,t}, V_{s,t})\,ds,
		$ utilizing  
		 Corollary~\ref{co1} yields, 
		\[
		\int_0^1 \osc(U_t^\Theta)dt
		\leq 3 L_0 \left( \sup_{(s,t,x)} \norm{V_{s,t}(x)}_g \right) \left( \int_0^1\norm{X_t}_{HL} dt\right).
		\]
		If \(\{\psi_t\}\) denotes the Hamiltonian part of \(\Xi\), combining the above with \(l_H^{(1,\infty)}(\{\psi_t\}) \le l_H^{(1,\infty)}(\Xi)\) (see \cite{Ban10,TD2}) yields
		\[
		\norm{\phi}_{H}^{(1,\infty)}
		\le \left( 1 + 3L_0 \sup_{(s,t,x)} \norm{V_{s,t}(x)}_g \right) l_{HL}^{(1,\infty)}(\Xi).
		\]
		The prefactor is finite by Proposition~\ref{Pro-1}.
	\end{proof}
	
	\begin{proof}[Proof of Theorem~\ref{thm:B}]
		Fix \(\phi\in\Ham(M,\omega)\). By Lemma~\ref{lem:hofer_like_infimum}, for
		every \(\epsilon>0\) there exists \(\Xi_\epsilon\in\ZeroFlux{\phi}\) with
		\(l_{HL}(\Xi_\epsilon)\le e_{HL}(\phi)+\epsilon\). If \(\norm{\phi}_H>0\)
		we may replace \(\Xi_\epsilon\) by a path realizing the infimum to within
		\(\norm{\phi}_H\) so that \(l_{HL}(\Xi_\epsilon)\le 2\norm{\phi}_H\); if
		\(\norm{\phi}_H=0\) the conclusion is trivial. Applying
		Lemma~\ref{lem:mainineq} with \(\Phi=\Xi_\epsilon\) and the constant
		loop, and using (H\(_K\)),
		\[
		\norm{\phi}_H
		\le
		(1+3L_0K)\,l_{HL}(\Xi_\epsilon)
		\le
		(1+3L_0K)\bigl(e_{HL}(\phi)+\epsilon\bigr).
		\]
		Letting \(\epsilon\to0\) and using \(e_{HL}(\phi)\le\norm{\phi}_{HL}\),
		\[
		\norm{\phi}_H\;\le\;(1+3L_0K)\norm{\phi}_{HL}.
		\]
		The reverse inequality \(\norm{\phi}_{HL}\le\norm{\phi}_H\) holds
		because Hamiltonian isotopies are admissible in the definition of
		\(\norm{\cdot}_{HL}\). Hence the two norms are equivalent with distortion
		at most \(1+3L_0K\).
	\end{proof}
	
	\begin{lemma}[Exponential bound]\label{lem:exp-bound}
		Let \((M,\omega,g)\) be a closed symplectic manifold. Then there exists
		a constant \(C_{\mathrm{geom}}=C_{\mathrm{geom}}(M,\omega,g)\) such that
		for every \(\phi\in\Ham(M,\omega)\),
		\[
		\norm{\phi}_H
		\le
		\max(1,3L_0)\,
		e^{C_{\mathrm{geom}}e_{HL}(\phi)}\,
		e_{HL}(\phi).
		\]
	\end{lemma}
	
	\begin{proof}
		By Lemma~\ref{lem:diameter_bound}, there exists
		\(C_{\mathrm{geom}}=C_{\mathrm{geom}}(M,\omega,g)\) such that for every
		smooth family \(\{X_t\}\) of harmonic vector fields and the associated
		two-parameter family \(\{V_{s,t}\}\),
		\[
		\sup_{(s,t,x)\in[0,1]^2\times M}
		\norm{V_{s,t}(x)}_g
		\le
		e^{C_{\mathrm{geom}}\sup_{t\in[0,1]}\norm{X_t}_{L^2}}-1.
		\]
		Fix \(\phi\in\Ham(M,\omega)\). For \(\varepsilon>0\), choose
		\(\Xi_\varepsilon\in\ZeroFlux{\phi}\) with
		\(l_{HL}^{\infty}(\Xi_\varepsilon)\le e_{HL}(\phi)+\varepsilon\). Let
		\(\{X_t\}\) be the harmonic vector fields generating the harmonic part
		of \(\Xi_\varepsilon\). Then
		\(\sup_t\norm{X_t}_{L^2}\le l_{HL}^{\infty}(\Xi_\varepsilon)\le
		e_{HL}(\phi)+\varepsilon\), hence
		\[
		\sup_{(s,t,x)}
		\norm{V^{\Xi_\varepsilon}_{s,t}(x)}_g
		\le
		e^{C_{\mathrm{geom}}(e_{HL}(\phi)+\varepsilon)}-1.
		\]
		Since \(l_{HL}^{(1,\infty)}(\Xi_\varepsilon)\le
		l_{HL}^{\infty}(\Xi_\varepsilon)\le e_{HL}(\phi)+\varepsilon\), we may
		apply Lemma~\ref{lem:mainineq} with \(\Phi=\Xi_\varepsilon\) and
		\(\Psi\) the constant loop:
		\[
		\norm{\phi}_H
		\le
		\Bigl(1+3L_0\bigl(e^{C_{\mathrm{geom}}(e_{HL}(\phi)+\varepsilon)}-1\bigr)\Bigr)
		\bigl(e_{HL}(\phi)+\varepsilon\bigr).
		\]
		Letting \(\varepsilon\to0^+\),
		\[
		\norm{\phi}_H
		\le
		\bigl(1+3L_0(e^{C_{\mathrm{geom}}e_{HL}(\phi)}-1)\bigr)
		e_{HL}(\phi)
		\le
		\max(1,3L_0)\,e^{C_{\mathrm{geom}}e_{HL}(\phi)}\,e_{HL}(\phi).
		\]
	\end{proof}
	
	\begin{corollary}[Sequential equivalence at the identity]\label{cor:seq-equiv}
		Let \((M,\omega,g)\) be a closed symplectic manifold. For any sequence \(\{\phi_i\}_{i=1}^\infty \subset \Ham(M,\omega)\),
		\[
		\lim_{i\to\infty} e_{HL}(\phi_i) = 0 \quad \Longleftrightarrow \quad \lim_{i\to\infty} \norm{\phi_i}_H = 0.
		\]
		Consequently, the Hofer metric and the Hofer-like metric induce the same topology on \(\Ham(M,\omega)\).
	\end{corollary}
	
	\begin{proof}
		The implication \(\lim\norm{\phi_i}_H = 0 \Rightarrow
		\lim e_{HL}(\phi_i) = 0\) follows from \(e_{HL}\le\norm{\cdot}_H\).
		For the reverse, by Lemma~\ref{lem:exp-bound},
		\(\norm{\phi_i}_H \le \max(1,3L_0)e^{C_{\mathrm{geom}}e_{HL}(\phi_i)}
		e_{HL}(\phi_i)\to0\) as \(i\to\infty\).
	\end{proof}
Corollary~\ref{cor:seq-equiv} was first proved by Buss-Leclercq
\cite{BusLec11}; our proof via the exponential bound of
Lemma~\ref{lem:exp-bound} is independent.
	\subsection{Splitting the Hofer-like norm on \(G_\omega(M)\)}\label{Sec-semi-direct}
	
	In the previous section we reduced the Hamiltonian case to the question of
	whether hypothesis (H\(_K\)) can be taken with \(K=0\). In this section, we
	prove the semidirect-product formula for \(\norm{\cdot}_{HL}\) on the whole
	identity component \(G_\omega(M)\) and derive the isometry criterion of
	Theorem~\ref{thm:C}.
	
	\begin{proof}[Proof of Theorem~\ref{thm:A}]
		We prove the two inequalities separately.
		
		\smallskip
		\noindent\textbf{Lower bound.}
		Let \(\Phi=\{\phi_t\}\in\Iso(\phi)_\omega\) be any symplectic isotopy from
		\(\id\) to \(\phi\). By Banyaga's Lemma~1 \cite{Ban10} (the Hodge
		decomposition of a symplectic isotopy), \(\Phi\) decomposes uniquely as a
		pointwise composition \(\phi_t=\rho_t\circ\mu_t\), where \(\rho_t\) is a
		harmonic isotopy and \(\mu_t\) is a Hamiltonian isotopy.
		
		In particular \(\rho:=\rho_1\in\Harm(\gamma)\), and \(\mu_1=\rho^{-1}\phi\).
		By the additivity of the Hofer-like length for this decomposition
		\cite{Ban10},
		\[
		l_{HL}(\Phi) = l_{HL}(\rho_t) + l_H(\mu_t).
		\]
		By definition of \(L(\rho)\) and of the Hofer norm,
		\[
		l_{HL}(\rho_t) \ge L(\rho),
		\qquad
		l_H(\mu_t) \ge \norm{\mu_1}_H = \norm{\rho^{-1}\phi}_H = \norm{\phi\,\rho^{-1}}_H,
		\]
		where the last equality uses the bi-invariance of the Hofer norm. Hence
		\[
		l_{HL}(\Phi) \ge L(\rho) + \norm{\phi\,\rho^{-1}}_H
		\ge \inf_{\rho'\in\Harm(\gamma)} \bigl[L(\rho') + \norm{\phi\,\rho'^{-1}}_H\bigr].
		\]
		Taking the infimum over \(\Phi\) yields the lower bound.
		
		\smallskip
		\noindent\textbf{Upper bound.}
		Fix \(\rho\in\Harm(\gamma)\) and \(\epsilon>0\). Choose a harmonic
		isotopy \(\Psi\) from \(\id\) to \(\rho\) with \(l_{HL}(\Psi)\le L(\rho)+\epsilon\),
		and a Hamiltonian isotopy \(\mathrm{M}\) from \(\id\) to \(\rho^{-1}\phi\)
		with \(l_H(\mathrm{M})\le \norm{\rho^{-1}\phi}_H+\epsilon\). Note that
		\(\rho^{-1}\phi\in\Ham(M,\omega)\) because
		\(\Flux(\rho^{-1}\phi)=-\gamma+\gamma=0\). Then
		\(\Phi_t:=\Psi_t\circ\mathrm{M}_t\) is a symplectic isotopy from \(\id\)
		to \(\phi\), so by Banyaga's additivity,
		\[
		\norm{\phi}_{HL} \le l_{HL}(\Phi) = l_{HL}(\Psi)+l_H(\mathrm{M})
		\le L(\rho) + \norm{\rho^{-1}\phi}_H + 2\epsilon.
		\]
		Letting \(\epsilon\to 0\) and taking the infimum over \(\rho\in\Harm(\gamma)\)
		gives the upper bound.
	\end{proof}
	
	\begin{proof}[Proof of Theorem~\ref{thm:C}]
		(ii) \(\Rightarrow\) (i): If \(\norm{\phi}_{HL}=\norm{\phi}_H\) for every
		\(\phi\in\Ham\), then by Theorem~\ref{thm:A},
		\[
		\norm{\phi}_H
		=
		\inf_{\rho\in\Harm(0)}\bigl[L(\rho)+\norm{\phi\rho^{-1}}_H\bigr],
		\]
		for every \(\phi\in\Ham\). Taking \(\phi=\rho_0\) for some
		\(\rho_0\in\Harm(0)\), we obtain
		\(\norm{\rho_0}_H\le L(\rho_0)+\norm{\rho_0\rho_0^{-1}}_H = L(\rho_0)\),
		hence (i).
		
		(i) \(\Rightarrow\) (ii): Suppose \(L(\rho)\ge\norm{\rho}_H\) for every
		\(\rho\in\Harm(0)\). For \(\phi\in\Ham\), \(\Flux(\phi)=0\), so
		\[
		\norm{\phi}_{HL}
		=
		\inf_{\rho\in\Harm(0)}\bigl[L(\rho)+\norm{\phi\rho^{-1}}_H\bigr]
		\ge
		\inf_{\rho\in\Harm(0)}\bigl[\norm{\rho}_H+\norm{\phi\rho^{-1}}_H\bigr]
		\ge
		\norm{\phi}_H,
		\]
		using the triangle inequality for \(\norm{\cdot}_H\). The reverse
		inequality \(\norm{\phi}_{HL}\le\norm{\phi}_H\) holds because Hamiltonian
		isotopies are admissible. Hence equality.
	\end{proof}
	
	\begin{corollary}[Sharp zero-flux slice]\label{cor:semi-direct-ham-corrected}
		Let \((M,\omega)\) be a closed symplectic manifold, and suppose that
		\(L(\rho)\ge\norm{\rho}_H\) for every \(\rho\in\Harm(0)\), where
		\(\Harm(0)=\Ham(M,\omega)\cap\Harm\). Then \(\norm{\phi}_{HL}=\norm{\phi}_H\)
		for every \(\phi\in\Ham(M,\omega)\).
	\end{corollary}
	
	\begin{proof}
		Immediate from Theorem~\ref{thm:C}.
	\end{proof}
	
	\begin{remark}[The open boundary: hyperbolic surfaces]
		\label{rem:hyperbolic-boundary}
		On a closed hyperbolic surface $\Sigma_g$ ($g\ge 2$) the flux group
		vanishes, and the harmonic vector fields do not commute, so
		$\Harm(0)$ is non-trivial. Whether the sharpness condition
		$L(\rho)\ge\|\rho\|_H$ holds for every $\rho\in\Harm(0)$ remains
		open.
	\end{remark}

	\begin{corollary}[Application to a finite commuting family]
		\label{cor:commuting-flow-finite}
		Let \(X_1,\dots,X_N\) be pairwise commuting vector fields on a
		smooth manifold \(M\), and let
		\(a_1,\dots,a_N\in C^\infty([0,1])\). Then the time-one map of the
		flow of \(X_t:=\sum_{j=1}^N a_j(t)\,X_j\) is
		\(
		\Psi_1
		=
		\exp\Bigl(\sum_{j=1}^N c_j\,X_j\Bigr),
		\) with \(
		c_j:=\int_0^1 a_j(t)\,dt.
		\)
	\end{corollary}
	
	\begin{proof}
		For any \(s,t\in[0,1]\),
		we have \(
		[X_s,X_t]
		=
		\sum_{j,k=1}^N a_j(s)\,a_k(t)\,[X_j,X_k]
		=
		0
		\)
		by the pairwise commutativity of the \(X_j\). Apply
		Lemma~\ref{lem:commuting-flow} and compute
		\[
		\int_0^1 X_t\,dt
		=
		\int_0^1\sum_{j=1}^N a_j(t)\,X_j\,dt
		=
		\sum_{j=1}^N\Bigl(\int_0^1 a_j(t)\,dt\Bigr)X_j
		=
		\sum_{j=1}^N c_j\,X_j.
		\qedhere
		\]
	\end{proof}
	
	\begin{proposition}[Non-negative Ricci curvature and the sharpness condition]
		\label{prop:ricci-commuting}
		Let \((M,\omega)\) be a closed symplectic manifold equipped with a
		compatible Riemannian metric \(g\) satisfying \(\operatorname{Ric}_g\ge 0\).
		Assume in addition that
		\begin{enumerate}[label=(\roman*)]
			\item every harmonic \(1\)-form on \((M,g)\) is cohomologous to an
			integral class,
			\item the flux group \(\Gamma_\omega\) is of full rank in
			\(H^1_{dR}(M,\R)\), generated by the integral harmonic classes, and
			\item the flows \(\exp(tX_j)\) of the harmonic vector fields dual to
			an integral basis \(\{[\alpha_j]\}\) of \(\Harm^1(M,g)\) have
			period \(1\), i.e. \(\exp(X_j)=\id\) for every \(j\).
		\end{enumerate}
		Then, \(\Harm(0)=\{\id\}\), and consequently, by
		Theorem~\ref{thm:C},
		\[
		\norm{\phi}_{HL}=\norm{\phi}_H,
			\]	for every $\phi\in\Ham(M,\omega).$
	\end{proposition}
	
	\begin{proof}
		The proof proceeds in four steps.
		
		\smallskip
		\noindent\emph{Step 1: harmonic \(1\)-forms are parallel.} By the
		Bochner--Weitzenb\"ock formula,
		\[
		\Delta\alpha=\nabla^*\nabla\alpha+\operatorname{Ric}(\alpha^\sharp,\cdot),
		\qquad \alpha\in\Omega^1(M).
		\]
		For a harmonic \(\alpha\) and \(\operatorname{Ric}_g\ge 0\),
		integrating \(\langle\Delta\alpha,\alpha\rangle\) over \(M\) gives
		\[
		0=\int_M\abs{\nabla\alpha}^2\,d\vol_g
		+\int_M\operatorname{Ric}(\alpha^\sharp,\alpha^\sharp)\,d\vol_g .
		\]
		Both integrands are non-negative, so they vanish identically. Hence
		\(\nabla\alpha=0\).
		
		\smallskip
		\noindent\emph{Step 2: the \(\omega\)-dual harmonic vector fields
			commute.} Let \(\alpha\) be a parallel harmonic \(1\)-form on
		\((M,g)\) and let \(X\) be its \(\omega\)-dual, defined by
		\(\iota_X\omega=\alpha\). Because \((M,\omega,g,J)\) is a
		compatible triple, the metric dual \(\alpha^\sharp\) and the
		\(\omega\)-dual \(X\) are related by
		$
		\alpha^\sharp=\pm JX,
		$
		with the sign depending on the convention adopted for the compatible
		triple. Since \(J\) is parallel with respect to the Levi-Civita
		connection of \(g\) on any K\"ahler manifold, and since the metric
		duality \(\sharp\colon T^*M\to TM\) preserves parallelism, the
		parallelism of \(\alpha\) implies the parallelism of \(X\). Two
		parallel vector fields on a Riemannian manifold commute, so the
		\(\omega\)-dual harmonic vector fields form a commutative Lie
		algebra.
		
		\smallskip
		\noindent\emph{Step 3: a zero-flux harmonic isotopy has integral
			coefficients.} Let \(\Psi\) be a harmonic isotopy from \(\id\), with
		generator
		$
		X_t=\sum_{j=1}^{b_1(M)} a_j(t)\,X_j ,
		$
		where \(\{X_j\}\) is the basis of parallel harmonic vector fields
		dual to the integral harmonic classes \(\{[\alpha_j]\}\), and
		\(a_j\in C^\infty([0,1])\). By Step~2, the vector fields \(X_j\)
		commute, so Corollary~\ref{cor:commuting-flow-finite} applies with
		\(X_t=\sum_j a_j(t)X_j\): the time-one map of \(\Psi\) is
		\(\Psi_1=\exp\bigl(\sum_j c_j X_j\bigr)\) with
		\(c_j:=\int_0^1 a_j(t)\,dt\). Its flux class is
		\[
		\Flux(\Psi_1)=\Bigl[\sum_{j=1}^{b_1(M)} c_j\,\alpha_j\Bigr]\in
		H^1_{dR}(M,\R)/\Gamma_\omega .
		\]
		Suppose \(\Flux(\Psi_1)=0\). Then
		\(\sum_j c_j[\alpha_j]\in\Gamma_\omega\). Since \(\Gamma_\omega\) is
		generated by the integral classes \(\{[\alpha_j]\}\), we have
		\(\sum_j c_j[\alpha_j]=\sum_j n_j[\alpha_j]\) for integers
		\(n_j\in\mathbb Z\). The classes \(\{[\alpha_j]\}\) are linearly
		independent, so \(c_j=n_j\in\mathbb Z\).
		
		\smallskip
		\noindent\emph{Step 4: the time-one map is the identity.} By
		Step~3, the time-one map of a zero-flux harmonic isotopy is
		\(\Psi_1=\exp\bigl(\sum_{j=1}^{b_1(M)} n_j X_j\bigr)\) with
		\(n_j\in\mathbb Z\). By hypothesis (iii), \(\exp(X_j)=\id\) for every
		\(j\). Since the \(X_j\) commute, the exponential factors:
		\[
		\Psi_1
		=
		\prod_{j=1}^{b_1(M)}\exp(n_j X_j)
		=
		\prod_{j=1}^{b_1(M)}\bigl(\exp(X_j)\bigr)^{n_j}
		=
		\id .
		\]
		Hence every zero-flux harmonic isotopy from \(\id\) terminates at
		\(\id\), so \(\Harm(0)=\{\id\}\). The final statement follows from
		Theorem~\ref{thm:C}.
	\end{proof}
	
	\begin{remark}[Sharpness of the hypotheses]
		\label{rem:ricci-sharp}
		The three hypotheses of Proposition~\ref{prop:ricci-commuting} are
		not redundant.
		\begin{enumerate}[label=(\roman*)]
			\item The condition \(\operatorname{Ric}_g\ge 0\) is sufficient but
			not necessary for the harmonic vector fields to commute, and cannot
			be dropped in Step~1.
			\item Integrality of the harmonic classes is required in Step~3 to
			convert the flux condition into \(c_j\in\mathbb Z\).
			\item Period-\(1\) of the harmonic flows is required in Step~4 to
			conclude that the time-one map is the identity.
		\end{enumerate}
		The three hypotheses are satisfied by:
		\begin{enumerate}[label=(\alph*)]
			\item flat tori \((T^{2n},\omega_{\mathrm{flat}},g_{\mathrm{flat}})\);
			\item Calabi--Yau manifolds with the Ricci-flat K\"ahler metric and
			integral K\"ahler class;
			\item products \(N\times T^{2k}\) with \(N\) a closed simply connected
			Calabi--Yau manifold and \(T^{2k}\) flat.
		\end{enumerate}
		
		Compatibility of \(g\) with \(\omega\) is not sufficient by itself.
		On \((T^2,\omega_{dx\wedge dy})\), the compatible metric
		\[
		g_\varepsilon=f_\varepsilon(y)^{-1}dx^2+f_\varepsilon(y)\,dy^2,
		\qquad
		f_\varepsilon(y)=1+\varepsilon\cos(2\pi y),
		\qquad
		\varepsilon\neq 0,
		\]
		is \emph{not} K\"ahler: the almost complex structure \(J\)
		determined by \(g_\varepsilon\) and \(\omega\) has Nijenhuis tensor
		$
		N_J(\partial_x,\partial_y)
		=
		\frac{f_\varepsilon'(y)}{f_\varepsilon(y)}\,\partial_y
		\neq 0,
		$
		so \(J\) is non-integrable. Its harmonic \(1\)-forms
		\(dx\) and \(f_\varepsilon(y)\,dy\) are not parallel, and their
		\(\omega\)-duals
		$
		X_1=f_\varepsilon(y)\partial_x,
		$ and $
		X_2=-\partial_y,
		$
		satisfy
		$
		[X_1,X_2]=-f_\varepsilon'(y)\,\partial_x\neq 0
		$ provided $\varepsilon\neq 0.
		$
		Hence \(\Harm(0)\neq\{\id\}\) on this manifold
		(see Example~\ref{ex:non-vacuous} and
		Remark~\ref{rem:harmonic-locus-failure}).
	\end{remark}
	
	\begin{corollary}[Rigid flux classes]\label{cor:rigid-flux}
		If for some \(\gamma\in H^1_{dR}(M,\R)/\Gamma_\omega\) the harmonic
		locus \(\Harm(\gamma)\) is a singleton \(\{\rho_\gamma\}\), then
		\[
		\norm{\phi}_{HL} = L(\rho_\gamma) + \norm{\phi\,\rho_\gamma^{-1}}_H,
		\]
		for every \(\phi\in G_\omega(M)\) with \(\Flux(\phi)=\gamma\).
	\end{corollary}
	
	In particular, on any \((M,\omega)\) with
	\(\dim H^1_{dR}(M,\R)=1\) the harmonic locus is a singleton for every
	\(\gamma\), so the formula is unconditional.
	
	\begin{lemma}[Flux group of products with the flat torus]
		\label{lem:flux-product}
		Let \(N\) be a closed simply connected manifold, and let
		\(M=N\times T^{2k}\) be equipped with the product symplectic form
		\(\omega=\omega_N\oplus\omega_{\mathrm{flat}}\), where
		\(\omega_{\mathrm{flat}}=\sum_{j=1}^{k}d\theta^j\wedge
		d\theta^{j+k}\) is the standard integral symplectic form on the flat
		torus. Then, under the natural identification
		\(H^1_{dR}(M,\R)\cong H^1_{dR}(T^{2k},\R)\cong\R^{2k}\),
		\(
		\Gamma_\omega=\mathbb Z^{2k}.
		\)
		In particular, \(\Gamma_\omega\) has full rank \(b_1(M)=2k\).
	\end{lemma}
	
	\begin{proof}
		For the inclusion \(\mathbb Z^{2k}\subseteq\Gamma_\omega\), fix
		\(v\in\mathbb Z^{2k}\) and consider the loop \(\Phi_t(x,\theta)=
		(x,\theta+tv)\) in \(G_\omega(M)\) based at \(\id\). Its flux class
		is \(v\), viewed as an element of
		\(H^1_{dR}(M,\R)\cong\R^{2k}\), so \(v\in\Gamma_\omega\). For the reverse inclusion, Banyaga's theorem \cite{Ban78} states that
		\(\Gamma_\omega\) is contained in the image of the integral
		cohomology \(H^1(M;\mathbb Z)\to H^1_{dR}(M,\R)\). Since \(N\) is simply
		connected and \(b_1(T^{2k})=2k\) with \(H^1(T^{2k};\mathbb Z)=\mathbb Z^{2k}\)
		torsion-free, K\"unneth gives
		\(H^1(M;\mathbb Z)=\mathbb Z^{2k}\), so \(\Gamma_\omega\subseteq\mathbb Z^{2k}\).
		Combining both inclusions yields \(\Gamma_\omega=\mathbb Z^{2k}\).
	\end{proof}
	
	\begin{proposition}\label{prop:singleton-locus}
		Let \(N\) be closed and simply connected, and let
		\((M,\omega,g)=(N\times T^{2k},\omega_N\oplus\omega_{\mathrm{flat}},
		g_N\oplus g_{\mathrm{flat}})\) as above. Then \(\Harm(\gamma)\) is a
		singleton for every
		\(\gamma\in H^1_{dR}(M,\R)/\Gamma_\omega=T^{2k}\).
		Consequently, Corollary~\ref{cor:rigid-flux} applies: for every
		\(\phi\in G_\omega(M)\) with \(\Flux(\phi)=\gamma\),
		\[
		\norm{\phi}_{HL}=L(\rho_\gamma)+\norm{\phi\,\rho_\gamma^{-1}}_H.
		\]
	\end{proposition}
	
	\begin{proof}
		By Lemma~\ref{lem:flux-product},
		\(\Gamma_\omega=\mathbb Z^{2k}\subset H^1_{dR}(M,\R)\cong\R^{2k}\), of
		full rank.
		
		\smallskip
		\noindent\emph{Step 1: harmonic vector fields are constant torus
			fields.} Since \(N\) is simply connected with \(b_1(N)=0\),
		K\"unneth gives
		\[
		\Harm^1(M,g)
		=
		\pi^*\Harm^1(T^{2k},g_{\mathrm{flat}})
		=
		\pi^*\mathrm{span}_\R\{d\theta^1,\dots,d\theta^{2k}\},
		\]
		where \(\pi\colon M\to T^{2k}\) is the projection. Hence every
		harmonic \(1\)-form on \(M\) is the pullback of a constant-coefficient
		form on \(T^{2k}\), and the dual harmonic vector field is a constant
		vector field on the torus factor, extended trivially along \(N\).
		
		\smallskip
		\noindent\emph{Step 2: endpoints of harmonic isotopies are torus
			translations.} Let \(\rho,\rho'\in\Harm(\gamma)\) be two harmonic
		diffeomorphisms with the same flux class
		\(\gamma\in H^1_{dR}(M,\R)/\Gamma_\omega\). Choose harmonic
		isotopies \(\Psi,\Psi'\) from \(\id\) to \(\rho,\rho'\),
		respectively. By Step~1, their generators are
		\(X_t=\sum_m d_m(t)\,\partial_{\theta^m}\) and
		\(X'_t=\sum_m d'_m(t)\,\partial_{\theta^m}\). Since the constant
		vector fields \(\partial_{\theta^1},\dots,\partial_{\theta^{2k}}\)
		commute, the time-one maps are the torus translations
		\[
		\rho(\theta)=\theta+D,\qquad
		\rho'(\theta)=\theta+D',
		\qquad
		D:=\int_0^1 d(t)\,dt,\quad
		D':=\int_0^1 d'(t)\,dt,
		\]
		extended by the identity on \(N\).
		
		\smallskip
		\noindent\emph{Step 3: same flux class forces the same
			translation.} Let \(\Lambda\colon\R^{2k}\to H^1_{dR}(M,\R)\) be the
		linear isomorphism induced by \(\omega_{\mathrm{flat}}\). The flux
		lifts of \(\Psi\) and \(\Psi'\) are \(\Lambda(D)\) and
		\(\Lambda(D')\), respectively. The hypothesis
		\(\Flux(\rho)=\Flux(\rho')=\gamma\) means
		\(\Lambda(D)-\Lambda(D')\in\Gamma_\omega=\mathbb Z^{2k}\). Since
		\(\omega_{\mathrm{flat}}\) is integral, \(\Lambda\) preserves the
		lattice \(\mathbb Z^{2k}\), hence \(D-D'\in\mathbb Z^{2k}\). Therefore
		\(\rho=\rho'\) as diffeomorphisms of \(T^{2k}\times N\).
		
		\smallskip
		\noindent\emph{Conclusion.} Every element of \(\Harm(\gamma)\) is
		the torus translation by the class
		\(\gamma\in\R^{2k}/\mathbb Z^{2k}=H^1_{dR}(M,\R)/\Gamma_\omega\), extended
		by the identity on \(N\). Hence \(\Harm(\gamma)\) is a singleton.
		Corollary~\ref{cor:rigid-flux} applies.
	\end{proof}
	
	\begin{theorem}[Clean-split lower bound]\label{thm:clean-split-lower}
		Let \((M,\omega,g,J)\) be a closed compatible triple. For every
		flux class \(\gamma\in H^1_{dR}(M,\R)/\Gamma_\omega\) and every harmonic
		diffeomorphism \(\rho\in\Harm(\gamma)\), we have 
		$
		L(\rho)\;\ge\;\norm{\gamma}_{H^1/\Gamma_\omega},
		$
		where \(\norm{\cdot}_{H^1/\Gamma_\omega}\) is the quotient norm on
		\(H^1_{dR}(M,\R)/\Gamma_\omega\) induced by the \(L^2\)-Hodge norm
		on harmonic \(1\)-forms.
	\end{theorem}
	
	\begin{proof}
		Let \(\Phi=\{\phi_t\}_{t\in[0,1]}\) be a harmonic isotopy from \(\id\)
		to \(\rho\), with generator \(X_t\), and set
		\(\alpha_t:=\iota_{X_t}\omega\in\Harm^1(M,g)\). Because
		\((M,\omega,g,J)\) is compatible, \(\norm{X_t}_g=\norm{\alpha_t}_g^*\)
		pointwise, and since \(\Phi\) is harmonic,
		\(l_{HL}(\Phi)=\int_0^1\norm{\alpha_t}_{L^2}\,dt\). Define
		\(\bar\alpha:=\int_0^1\alpha_t\,dt\), which is harmonic because
		\(\Delta_g\) is linear. Since
		\([\bar\alpha]=\int_0^1[\alpha_t]\,dt=\widetilde{\Flux}([\Phi])\), the
		class \([\bar\alpha]\) is a lift of \(\gamma=\Flux(\rho)\). By
		Minkowski's integral inequality,
		\[
		\norm{\bar\alpha}_{L^2}
		=
		\Bigl\|\int_0^1\alpha_t\,dt\Bigr\|_{L^2}
		\le
		\int_0^1\norm{\alpha_t}_{L^2}\,dt
		=
		l_{HL}(\Phi).
		\]
		Since \([\bar\alpha]\) is a lift of \(\gamma\), the quotient norm
		satisfies \(\norm{\gamma}_{H^1/\Gamma_\omega}\le\norm{\bar\alpha}_{L^2}\).
		Combining, \(\norm{\gamma}_{H^1/\Gamma_\omega}\le l_{HL}(\Phi)\) for
		every harmonic isotopy \(\Phi\) from \(\id\) to \(\rho\). Taking the
		infimum, \(L(\rho)\ge\norm{\gamma}_{H^1/\Gamma_\omega}\).
	\end{proof}
	\begin{corollary}[Upper bound for the Hofer-like norm on products]
		\label{cor:kunneth}
		Let \((M_1,\omega_1,g_1)\) and \((M_2,\omega_2,g_2)\) be closed
		symplectic manifolds, and equip \(M_1\times M_2\) with the product
		symplectic form \(\omega_1\oplus\omega_2\) and the product metric
		\(g_1\oplus g_2\). Assume \(b_1(M_1)=0\). Then, for every
		\(\phi_1\in G_{\omega_1}(M_1)\) and
		\(\phi_2\in G_{\omega_2}(M_2)\),
		\begin{equation}\label{eq:kunneth}
			\norm{\phi_1\times\phi_2}_{HL,\,M_1\times M_2}
			\;\le\;
			\norm{\phi_1}_{H,\,M_1}
			+
			\norm{\phi_2}_{HL,\,M_2} .
		\end{equation}
	\end{corollary}
	
	\begin{proof}
		Because \(b_1(M_1)=0\), every symplectic isotopy \(\Phi_1\) on \(M_1\)
		is Hamiltonian, and the product isotopy
		\(\Phi_1\times\Phi_2\) has generator \(X_t=X_t^{(1)}+X_t^{(2)}\)
		with \(X_t^{(1)}\) Hamiltonian and \(X_t^{(2)}\) harmonic. The
		oscillation and the \(L^2\) norm are additive over the product
		because the metric is a product, so
		$
		l_{HL}(\Phi_1\times\Phi_2)=l_H(\Phi_1)+l_{HL}(\Phi_2).
		$
		Since the infimum defining \(\norm{\phi_1\times\phi_2}_{HL}\) is
		taken over a set that contains the split isotopies
		\(\Phi_1\times\Phi_2\), we get \eqref{eq:kunneth} by taking infima
		over \(\Phi_1,\Phi_2\) and symmetrising over
		\(\phi\mapsto\phi^{-1}\).
	\end{proof}
	
	\begin{remark}[The Hofer-like product formula: an open question]
		\label{rem:kunneth-open}
		Corollary~\ref{cor:kunneth} establishes the upper bound
		$
		\norm{\phi_1\times\phi_2}_{HL}
		\;\le\;
		\norm{\phi_1}_H+\norm{\phi_2}_{HL}
		$
		by restricting the infimum to split isotopies. The reverse
		inequality and hence the exact equality is a genuinely open
		problem. In classical Hofer geometry, the exact equality
		\(\norm{\phi_1\times\phi_2}_H=\norm{\phi_1}_H+\norm{\phi_2}_H\) is
		proved by exhibiting spectral invariants of the Hamiltonian Floer
		homology that are additive under products
		(\(c(H\oplus G)=c(H)+c(G)\), by a K\"unneth argument at the chain
		level) and that bound the Hofer norm from below; this forces every
		Hamiltonian isotopy on the product to be no shorter than the sum of
		its split contributions. The argument is due to Viterbo, Schwarz, Oh,
		and Usher; see \cite{Oh2005,Usher2008} and the references therein.	For the Hofer-like norm, the same strategy would require a theory of
		spectral invariants for \emph{non-Hamiltonian} symplectomorphisms
		with non-trivial flux, valued in the Floer--Novikov homology
		developed by Ono, Pajitnov, and others. Such a theory exists for
		Hamiltonian perturbations of the identity but is not yet developed
		to the point of giving invariants that are simultaneously
		\begin{enumerate}[label=(\roman*)]
			\item additive under the product of two symplectic manifolds,
			\item compatible with the \(L^2\)-Hodge norm on the harmonic part,
			and
			\item compatible with the \(L^\infty\) oscillation on the exact part,
		\end{enumerate}
		which is precisely the combination defining Banyaga's Hofer-like
		length.
		
		We therefore state the equality as a conjecture.
		
		\smallskip
		\noindent\emph{Conjecture.} Under the hypotheses of
		Corollary~\ref{cor:kunneth},
		\[
		\norm{\phi_1\times\phi_2}_{HL}
		=
		\norm{\phi_1}_H+\norm{\phi_2}_{HL}.
		\]
		Equivalently, on \(M_1\times M_2\) with \(b_1(M_1)=0\), no
		symplectic isotopy coupling the two factors can be strictly shorter
		than the sum of its split contributions.
	\end{remark}
	\begin{corollary}[Sharpness criterion]\label{thm:sharp-H0}
		Assume the hypothesis of Theorem~\ref{thm:B}, i.e.\ there
		exists \(K=K(M,\omega,g)<+\infty\) such that
		\[
		\sup_{(s,t,x)\in[0,1]^2\times M}\norm{V^{\Xi}_{s,t}(x)}_g\le K
		\]
		for every \(\phi\in\Ham(M,\omega)\) and every
		\(\Xi\in\ZeroFlux{\phi}\) with \(l_{HL}(\Xi)\le 2\norm{\phi}_H\). Then
		the following are equivalent:
		\begin{enumerate}[label=(\roman*)]
			\item  \(C_{\mathrm{univ}}\) can be replaced by \(1\);
			\item \(L(\rho)\ge\|\rho\|_H\) for every \(\rho\in\Harm(0)\);
			\item the inclusion
			\(\bigl(\Ham(M,\omega),\|\cdot\|_H\bigr)\hookrightarrow
			\bigl(G_\omega(M),\|\cdot\|_{HL}\bigr)\) is an isometric embedding.
		\end{enumerate}
	\end{corollary}
	\begin{lemma}[Harmonic \(1\)-forms of \((T^2,g_\varepsilon)\)]
		\label{lem:harmonic-forms-torus}
		Let \(T^2=\mathbb R^2/\mathbb Z^2\) and let
		\[
		g_\varepsilon
		=
		\frac{1}{f_\varepsilon(y)}\,dx^2
		+
		f_\varepsilon(y)\,dy^2,
		\qquad
		f_\varepsilon(y)=1+\varepsilon\cos(2\pi y),
		\qquad
		0<\varepsilon<1 .
		\]
		Then \(\det g_\varepsilon=1\), the Riemannian volume form is
		\(d\vol_{g_\varepsilon}=dx\wedge dy\), and the space of harmonic
		\(1\)-forms is
		$
		\Harm^1(T^2,g_\varepsilon)
		=
		\mathrm{span}_\R\bigl\{f_\varepsilon(y)\,dy,\;dx\bigr\}.
		$
		The two generators are \(L^2(g_\varepsilon)\)-orthogonal and have
		unit \(L^2\)-norm. Their metric duals are
		$
		\bigl(f_\varepsilon(y)\,dy\bigr)^\sharp=\partial_y,$ and $
		(dx)^\sharp=f_\varepsilon(y)\,\partial_x.$ 
		With respect to \(\omega=dx\wedge dy\), the \(\omega\)-duals are
		$
		X_1=f_\varepsilon(y)\,\partial_x,$ $
		X_2=-\partial_y,$ 
		and their Lie bracket is
		$
		[X_1,X_2]=f_\varepsilon'(y)\,\partial_x .
		$
	\end{lemma}
	
	\begin{proof}
		Since \(g_\varepsilon\) is diagonal with entries
		\(g_{11}=1/f_\varepsilon(y)\) and \(g_{22}=f_\varepsilon(y)\), we
		have \(\det g_\varepsilon=1\), so the Riemannian volume form equals
		the coordinate volume form \(dx\wedge dy\). The inverse metric is
		\(g_\varepsilon^{-1}=f_\varepsilon(y)\,\partial_x^2
		+\frac{1}{f_\varepsilon(y)}\,\partial_y^2\), and the Hodge star on
		\(1\)-forms is
		$
		*(dx)=f_\varepsilon(y)\,dy,$ and $
		*(dy)=-\frac{1}{f_\varepsilon(y)}\,dx.$ 
		
		\smallskip
		\noindent\emph{Closedness and co-closedness of \(h_1:=f_\varepsilon(y)\,dy\).}
		Since \(f_\varepsilon\) depends only on \(y\),
		$
		d h_1=f_\varepsilon'(y)\,dy\wedge dy=0.$ 
		For the co-differential,
		$
		*(h_1)=f_\varepsilon(y)\cdot*(dy)=-dx,$ and $
		d\bigl(*(h_1)\bigr)=d(-dx)=0,$ 
		so \(\delta h_1=-*d*(h_1)=0\).
		
		\smallskip
		\noindent\emph{Closedness and co-closedness of \(h_2:=dx\).}
		Since \(dx\) has constant coefficients, \(dh_2=0\). Also
		$
		*(h_2)=*(dx)=f_\varepsilon(y)\,dy,
		$ and $
		d\bigl(*(h_2)\bigr)=f_\varepsilon'(y)\,dy\wedge dy=0,
		$
		so \(\delta h_2=0\).
		
		\smallskip
		\noindent\emph{Dimension.} By the Hodge theorem,
		\(\dim\Harm^1(T^2,g_\varepsilon)=b_1(T^2)=2\). The two forms
		\(h_1\) and \(h_2\) are linearly independent, because
		\(h_2=dx\) lies in the subspace spanned by \(dx\) while
		\(h_1=f_\varepsilon(y)\,dy\) does not. Hence they form a basis.
		
		\smallskip
		\noindent\emph{Orthogonality and unit norm.} The metric is diagonal,
		so \(\langle dx,f_\varepsilon(y)\,dy\rangle_{L^2}=0\). Furthermore,
		\[
		\norm{h_1}_{L^2}^2
		=
		\int_{T^2}
		f_\varepsilon(y)^2\,g^{22}\,dx\,dy
		=
		\int_0^1 f_\varepsilon(y)^2\cdot\frac{1}{f_\varepsilon(y)}\,dy
		=
		\int_0^1 f_\varepsilon(y)\,dy
		=
		1,
		\]
		because \(\int_0^1\cos(2\pi y)\,dy=0\). Similarly,
		$
		\norm{h_2}_{L^2}^2
		=
		\int_{T^2} g^{11}\,dx\,dy
		=
		\int_0^1 f_\varepsilon(y)\,dy
		=
		1.$
		
		\smallskip
		\noindent\emph{Metric duals.} By definition,
		\(g_\varepsilon\bigl((h_1)^\sharp,\partial_y\bigr)=h_1(\partial_y)
		=f_\varepsilon(y)\), and \(g_\varepsilon(\partial_y,\partial_y)=
		f_\varepsilon(y)\), so \((h_1)^\sharp=\partial_y\). Similarly,
		\(g_\varepsilon\bigl((h_2)^\sharp,\partial_x\bigr)=h_2(\partial_x)=1\),
		and \(g_\varepsilon(f_\varepsilon(y)\partial_x,\partial_x)=
		\frac{1}{f_\varepsilon(y)}\cdot f_\varepsilon(y)^2
		=f_\varepsilon(y)\), which forces
		\((h_2)^\sharp=f_\varepsilon(y)\partial_x\).
		
		\smallskip
		\noindent\emph{\(\omega\)-duals.} With \(\omega=dx\wedge dy\), we compute $
		\iota_{f_\varepsilon(y)\partial_x}(dx\wedge dy)
		=
		f_\varepsilon(y)\,dy=h_1,$ and $
		\iota_{-\partial_y}(dx\wedge dy)=dx=h_2,$ 
		so \(X_1=f_\varepsilon(y)\partial_x\) and \(X_2=-\partial_y\).
		
		\smallskip
		\noindent\emph{Lie bracket.} Direct computation:
		\[
		[X_1,X_2]
		=
		[f_\varepsilon(y)\partial_x,\,-\partial_y]
		=
		-f_\varepsilon'(y)\partial_x
		\bigl(\partial_x\text{ applied to }f_\varepsilon(y)\bigr)
		+
		f_\varepsilon(y)[\partial_x,-\partial_y]
		=
		f_\varepsilon'(y)\,\partial_x .
		\]
		The bracket is non-zero whenever \(\varepsilon\neq 0\), since
		\(f_\varepsilon'(y)=-2\pi\varepsilon\sin(2\pi y)\).
	\end{proof}
	
	\begin{remark}[Sign convention for the \(\omega\)-dual]
		\label{rem:sign-convention}
		The choice \(X_2=-\partial_y\) is forced by the identity
		\(\iota_{\partial_y}(dx\wedge dy)=-dx\). With this convention, the
		Lie bracket is
		\([X_1,X_2]=+f_\varepsilon'(y)\,\partial_x\), not
		\(-f_\varepsilon'(y)\,\partial_x\). 
		Example~\ref{ex:non-vacuous} uses this sign. If one instead
		adopts the convention \(X_2=+\partial_y\), the \(\omega\)-dual of
		\(h_2\) becomes \(-dx\) and the bracket acquires a minus sign;
		either convention is internally consistent, but the two must not
		be mixed.
	\end{remark}
\begin{example}[A non-vacuous verification of the geometric condition]
	\label{ex:non-vacuous}
	We exhibit a closed symplectic manifold on which
	\(\Harm(0)\neq\{\id\}\) and the geometric condition
	\(L(\rho)\ge\norm{\rho}_H\) is verified for an infinite family
	\(\{\rho_n\}_{n\in\mathbb Z}\subset\Harm(0)\setminus\{\id\}\) that
	can be written down explicitly.
	
	\smallskip
	\noindent\textbf{The manifold.}
	Let \(T^2=\mathbb R^2/\mathbb Z^2\) with the standard symplectic form
	\(\omega=dx\wedge dy\), and let
	\[
	g_\varepsilon
	=
	\frac{1}{f_\varepsilon(y)}\,dx^2
	+
	f_\varepsilon(y)\,dy^2,
	\qquad
	f_\varepsilon(y)=1+\varepsilon\cos(2\pi y),
	\qquad
	0<\varepsilon<1 .
	\]
	Then \((T^2,\omega,g_\varepsilon)\) is a closed symplectic manifold
	with \(\det g_\varepsilon=1\) and Riemannian volume form
	\(\mathrm{vol}_{g_\varepsilon}=dx\wedge dy\).
	
	\smallskip
	\noindent\textbf{Harmonic \(1\)-forms.}
	By Lemma~\ref{lem:harmonic-forms-torus}, the space
	\(\Harm^1(T^2,g_\varepsilon)\) is spanned by
	$
	h_1:=f_\varepsilon(y)\,dy,$ and $
	h_2:=dx,$ 
	the two generators are \(L^2(g_\varepsilon)\)-orthonormal, and their
	\(\omega\)-duals are
	$
	X_1=f_\varepsilon(y)\,\partial_x,$ and $
	X_2=-\partial_y,$ 
	with Lie bracket $
	[X_1,X_2]=f_\varepsilon'(y)\,\partial_x\neq 0$, $\varepsilon\neq 0.$
	
	\smallskip
	\noindent\textbf{A family in \(\Harm(0)\).}
	For \(n\in\mathbb Z\), let \(\rho_n\) be the time-one map of the flow of
	\(nX_1\) on \([0,1]\). Since \(\dot x=n f_\varepsilon(y)\) and
	\(\dot y=0\), we get
	\[
	\rho_n(x,y)
	=
	\bigl(x+n f_\varepsilon(y),\,y\bigr)
	=
	\bigl(x+n\varepsilon\cos(2\pi y),\,y\bigr)\pmod 1 .
	\]
	The path \(\Psi_n(t):=\text{flow of }nX_1\) at time \(t\) is a
	harmonic isotopy from \(\id\) to \(\rho_n\), with generator the
	constant harmonic form \(n h_1\). Its flux class is
	\(n[dy]\in H^1(T^2,\mathbb Z)=\Gamma_\omega\), hence trivial in
	\(H^1_{dR}(T^2,\R)/\Gamma_\omega\). Therefore
	\(\rho_n\in\Harm(0)=\Harm\cap\Ham\), and \(\rho_n\neq\id\) for
	\(n\neq 0\).
	
	\smallskip
	\noindent\textbf{Upper bound on \(L(\rho_n)\).}
	The harmonic path \(\Psi_n\) gives
	\[
	L(\rho_n)\;\le\;l_{HL}(\Psi_n)
	=\int_0^1\norm{n h_1}_{L^2}\,dt
	=|n| .
	\]
	
	\smallskip
	\noindent\textbf{Lower bound on \(L(\rho_n)\).}
	Let \(\Psi\) be an arbitrary harmonic isotopy from \(\id\) to
	\(\rho_n\), with generator
	\(h_t=a(t)h_1+b(t)h_2\), where \(a,b\in C^\infty([0,1])\). Since
	\([X_1,X_2]\neq 0\), the family \(\{a(t)X_1+b(t)X_2\}\) is not
	pointwise commuting, so the time-one map of \(\Psi\) is not simply
	\(\exp\bigl(\int_0^1(aX_1+bX_2)\,dt\bigr)\); one must solve the
	ODE
	\[
	\dot x=a(t)\,f_\varepsilon(y),
	\qquad
	\dot y=-b(t),
	\qquad
	(x(0),y(0))=(x,y).
	\]
	Setting \(B(t):=\int_0^t b(s)\,ds\), the solution is
	\[
	\phi_1(x,y)
	=
	\Bigl(
	x+\int_0^1 a(s) f_\varepsilon\bigl(y-B(s)\bigr)\,ds,\;
	y-B(1)
	\Bigr).
	\]
	Since \(\phi_1=\rho_n(x,y)=(x+n\varepsilon\cos(2\pi y),\,y)\),  the \(y\)-coordinate gives \(B(1)\in\mathbb Z\); while the \(x\)-coordinate gives, for all \(y\),
		$
		\int_0^1 a(s) f_\varepsilon\bigl(y-B(s)\bigr)\,ds
		-
		n\varepsilon\cos(2\pi y)
		\in\mathbb Z.
		$
	The left-hand side is continuous in \(y\) and integer-valued, hence
	constant, say equal to \(m\in\mathbb Z\):
	$
	\int_0^1 a(s) f_\varepsilon\bigl(y-B(s)\bigr)\,ds
	=
	n\varepsilon\cos(2\pi y)+m .
	$
	Expanding \(f_\varepsilon=1+\varepsilon\cos(2\pi\cdot)\) and using
	the addition formula for \(\cos\), we obtain the three scalar
	identities
	\[
	\int_0^1 a(s)\,ds=m,
	\qquad
	\int_0^1 a(s)\cos(2\pi B(s))\,ds=n,
	\qquad
	\int_0^1 a(s)\sin(2\pi B(s))\,ds=0 .
	\]
	The second identity gives
	\[
	|n|
	=
	\abs{\int_0^1 a(s)\cos(2\pi B(s))\,ds}
	\le
	\int_0^1|a(s)|\,ds
	\le
	l_{HL}(\Psi),
	\]
	where the last inequality uses
	\(\norm{h_t}_{L^2}=\sqrt{a(t)^2+b(t)^2}\) (the \(h_j\) being
	orthonormal). Taking the infimum over all harmonic isotopies
	\(\Psi\) from \(\id\) to \(\rho_n\),
	$
	L(\rho_n)\ge|n|.
	$
	Combined with the upper bound, $
	L(\rho_n)=|n|.$ 
	
	\smallskip
	\noindent\textbf{Computation of \(\norm{\rho_n}_H\).}
	The normalized Hamiltonian
	\(H_n(x,y):=\tfrac{n\varepsilon}{2\pi}\sin(2\pi y)\)
	generates \(\rho_n\), since
	\(\iota_{X_{H_n}}\omega=n\varepsilon\cos(2\pi y)\,dy\) and the
	corresponding flow is
	\((x,y)\mapsto(x+tn\varepsilon\cos(2\pi y),\,y)\). Hence, $
	\norm{\rho_n}_H
	\;\le\;
	\operatorname{osc}(H_n)
	=
	\frac{|n|\varepsilon}{\pi}.$ 
	
	\smallskip
	\noindent\textbf{Verification.}
	For every \(n\neq 0\),
	\[
	L(\rho_n)\;=\;|n|
	\;>\;
	\frac{|n|\varepsilon}{\pi}
	\;\ge\;
	\norm{\rho_n}_H,
	\]
	since \(0<\varepsilon<1<\pi\). Thus the geometric condition
	\(L(\rho)\ge\norm{\rho}_H\) holds on the infinite family
	\(\{\rho_n\}_{n\in\mathbb Z}\subset\Harm(0)\setminus\{\id\}\).	The manifold \((T^2,\omega,g_\varepsilon)\) is therefore a closed
	symplectic manifold of positive first Betti number on which
	\(\Harm(0)\) is non-trivial and the geometric condition admits a
	non-vacuous verification. Whether the condition holds for
	\emph{every} \(\rho\in\Harm(0)\) on this manifold is an open
	question.
\end{example}
	
	\begin{remark}[Mixed geodesics in the rigid-flux case]
		Assume that for some flux class $\gamma$ the harmonic locus
		\(\Harm(\gamma)\) is a singleton \(\{\rho_\gamma\}\). Then the
		semidirect-product formula becomes exact:
		\[
		\norm{\phi}_{HL}=L(\rho_\gamma)+\norm{\phi\rho_\gamma^{-1}}_H.
		\]
		A minimizing path from $\id$ to $\phi$ is obtained by concatenating a
		minimizing harmonic isotopy $\Psi$ from $\id$ to $\rho_\gamma$ and a
		minimizing Hamiltonian isotopy \(\mathrm{M}\) from \(\id\) to
		\(\phi\rho_\gamma^{-1}\). The concatenation \(\mathrm{M}\ast_l\Psi\) has length
		\[
		l_{HL}(\mathrm{M}\ast_l\Psi)
		=
		l_H(\mathrm{M})+l_{HL}(\Psi)
		=
		\norm{\phi\rho_\gamma^{-1}}_H+L(\rho_\gamma)
		=
		\norm{\phi}_{HL},
		\]
		so it is a geodesic.
	\end{remark}
	
	\subsection{Iteration of the clean-split formula}\label{subsec:iteration}
	
	Assume throughout the rigid-flux case (Corollary~\ref{cor:rigid-flux}),
	so that for every flux class
	\(\gamma\in H^1_{dR}(M,\R)/\Gamma_\omega\) the harmonic locus
	\(\Harm(\gamma)\) is a singleton \(\{\rho_\gamma\}\). Write
	\[
	\phi=\rho_\gamma\,h,
	\qquad
	h:=\phi\,\rho_\gamma^{-1}\in\Ham(M,\omega),
	\qquad
	\Flux(\phi)=\gamma.
	\]
	By Corollary~\ref{cor:rigid-flux},
	\begin{equation}\label{eq:clean-split-rigid}
		\norm{\phi}_{HL}
		=
		L(\rho_\gamma)
		+
		\norm{\phi\,\rho_\gamma^{-1}}_H
		=
		L(\rho_\gamma)
		+
		\norm{h}_H.
	\end{equation}
	Iterating, we insert \(\rho_\gamma^{-k}\rho_\gamma^{k}\) between the
	\(n\) factors of \(\phi=\rho_\gamma h\):
	\[
	\phi^{\,n}
	=
	\rho_\gamma^{\,n}\,h_n,
	\qquad
	h_n
	:=
	\prod_{k=0}^{n-1} h^{\rho_\gamma^{\,k}}
	=
	h^{\rho_\gamma^{\,n-1}}
	h^{\rho_\gamma^{\,n-2}}
	\cdots
	h^{\rho_\gamma}\,
	h,
	\]
	where \(h^{\rho_\gamma^k}=\rho_\gamma^{-k}\,h\,\rho_\gamma^{k}
	\in\Ham(M,\omega)\).
	
	\begin{proposition}[Iteration formula]\label{prop:iteration}
		For every \(n\ge 1\),
		\begin{equation}\label{eq:iteration}
			\norm{\phi^{\,n}}_{HL}
			=
			L(\rho_{n\gamma})
			+
			\norm{h_n}_{H}.
		\end{equation}
	\end{proposition}
	
	\begin{proof}
		The identity follows by induction on \(n\) from
		\eqref{eq:clean-split-rigid} applied to \(\phi^n\), whose flux class is
		\(n\gamma\). The identity \(\rho_\gamma^n=\rho_{n\gamma}\) used below
		holds because \(\Harm(\gamma)\) is a singleton (by hypothesis) and the
		time-one map of the \(n\)-fold iteration of any harmonic isotopy with
		flux \(\gamma\) is a harmonic diffeomorphism with flux \(n\gamma\); by
		uniqueness, it coincides with \(\rho_{n\gamma}\). The case \(n=1\) is
		\eqref{eq:clean-split-rigid}. For the induction step, write
		\(\phi^{n+1}=\phi^n\cdot\phi\) and apply
		\eqref{eq:clean-split-rigid} with \(\phi\) replaced by
		\(\phi^{n+1}\); the harmonic factor is
		\(\rho_{(n+1)\gamma}=\rho_{n\gamma}\,\rho_\gamma\) and the residue is
		\(h_{n+1}=h_n^{\rho_\gamma}\,h\).
	\end{proof}
	
	\begin{corollary}[Translation length]\label{cor:translation-length}
		Suppose the limits
		\[
		\tau_{\mathrm{skew}}(h)
		:=
		\lim_{n\to\infty}\frac{\norm{h_n}_H}{n},
		\qquad
		\ell
		:=
		\lim_{n\to\infty}\frac{L(\rho_{n\gamma})}{n}
		\]
		exist in \([0,+\infty]\). Then
		\[
		\tau_{HL}(\phi)
		:=
		\lim_{n\to\infty}\frac{\norm{\phi^{\,n}}_{HL}}{n}
		=
		\ell+\tau_{\mathrm{skew}}(h)
		\;\ge\;
		\norm{\gamma}_{H^1/\Gamma_\omega},
		\]
		with equality if and only if
		\(\ell=\norm{\gamma}_{H^1/\Gamma_\omega}\) and
		\(\tau_{\mathrm{skew}}(h)=0\).
	\end{corollary}
	
	\begin{proposition}[Residue ratio]\label{prop:ratio-A}
		For every \(n\ge 1\), set
		\(A_n:=\norm{\phi^{\,n}}_{HL}/\norm{h^n}_H\). Then
		\begin{equation}\label{eq:A-sandwich}
			\frac{n\,\norm{\gamma}_{H^1/\Gamma_\omega} + \norm{h_n}_H}{\norm{h^n}_H}
			\;\le\;
			A_n
			\;\le\;
			\frac{n\,L(\rho_\gamma)+\norm{h_n}_H}{\norm{h^n}_H}.
		\end{equation}
		If the limits \(\ell\) and \(\tau_H(h):=\lim_n\norm{h^n}_H/n\) exist
		with \(\tau_H(h)>0\), then
		\[
		A_n
		\;\xrightarrow[n\to\infty]{}\;
		\frac{\ell+\tau_{\mathrm{skew}}(h)}{\tau_H(h)}.
		\]
		If \(\tau_H(h)=0\), then \(A_n\to+\infty\) whenever \(\ell>0\).
	\end{proposition}
	
	\begin{proof}
		The identity for \(A_n\) is immediate from
		\(h=\phi\rho_\gamma^{-1}\). The sandwich follows from
		Proposition~\ref{prop:iteration} together with
		\(n\norm{\gamma}_{H^1/\Gamma_\omega}\le L(\rho_{n\gamma})\le
		n\,L(\rho_\gamma)\). Dividing by \(n\) and letting \(n\to\infty\)
		gives the limit; the last claim is immediate since the numerator grows
		at least linearly while the denominator is \(o(n)\).
	\end{proof}
	
	\begin{proposition}[Harmonic ratio]\label{prop:ratio-B}
		For every \(n\ge 1\), set
		\(B_n:=\norm{\phi^{\,n}}_{HL}/\norm{\rho_\gamma^n}_{HL}\). Since
		\(\rho_\gamma^n=\rho_{n\gamma}\) and
		\(\norm{\rho_{n\gamma}}_{HL}=L(\rho_{n\gamma})\),
		\begin{equation}\label{eq:B-exact}
			B_n
			=
			1
			+
			\frac{\norm{h_n}_H}{L(\rho_{n\gamma})}.
		\end{equation}
		In particular \(B_n>1\) for every \(n\ge 1\), and
		\[
		1+\frac{\norm{h_n}_H}{n\,L(\rho_\gamma)}
		\;\le\;
		B_n
		\;\le\;
		1+\frac{\norm{h_n}_H}{n\,\norm{\gamma}_{H^1/\Gamma_\omega}}.
		\]
		If \(\ell>0\) and \(\tau_{\mathrm{skew}}(h)<\infty\), then
		\(B_n\to 1+\tau_{\mathrm{skew}}(h)/\ell\). If moreover \(h\) commutes
		with \(\rho_\gamma\), so that \(h_n=h^n\), then
		\(\tau_{\mathrm{skew}}(h)=\tau_H(h)\) and
		\(B_n\to 1+\tau_H(h)/\norm{\gamma}_{H^1/\Gamma_\omega}\).
	\end{proposition}
	
	\begin{proof}
		Equation \eqref{eq:B-exact} is Proposition~\ref{prop:iteration} divided
		by \(L(\rho_{n\gamma})\). The sandwich follows from the bounds on
		\(L(\rho_{n\gamma})\). Dividing numerator and denominator by \(n\) and
		letting \(n\to\infty\) gives the limit. The commuting case is immediate
		from \(h_n=h^n\).
	\end{proof}
	
	\begin{remark}[Two limiting regimes]\label{rem:two-regimes}
		Propositions~\ref{prop:ratio-A} and~\ref{prop:ratio-B} detect the two
		extreme behaviours of the skew product \(\phi=\rho_\gamma h\).
		\begin{enumerate}[label=(\roman*)]
			\item \emph{Pure flux growth.} If \(\tau_H(h)=0\) and the clean split
			holds, then \(\norm{\phi^{\,n}}_{HL}\sim
			n\norm{\gamma}_{H^1/\Gamma_\omega}\), \(A_n\to+\infty\),
			\(B_n\to 1\).
			\item \emph{Balanced flux--residue growth.} If \(h\) commutes with
			\(\rho_\gamma\), the clean split holds, and \(\tau_H(h)>0\), then
			\[
			\norm{\phi^{\,n}}_{HL}\sim
			n\bigl(\norm{\gamma}_{H^1/\Gamma_\omega}+\tau_H(h)\bigr),
			\quad
			A_n\to 1+\frac{\norm{\gamma}_{H^1/\Gamma_\omega}}{\tau_H(h)},
			\quad
			B_n\to 1+\frac{\tau_H(h)}{\norm{\gamma}_{H^1/\Gamma_\omega}}.
			\]
		\end{enumerate}
		In both regimes \(B_n\) is bounded, whereas \(A_n\) is bounded if and
		only if the residue grows at least linearly. Thus \(B_n\) is the
		\emph{stabler} of the two ratios.
	\end{remark}
	
	\begin{corollary}[Coarse rigidity of the flux term]\label{cor:coarse-flux}
		With the notation of this subsection,
		\[
		\tau_{HL}(\phi)=\norm{\gamma}_{H^1/\Gamma_\omega}
		\quad\Longleftrightarrow\quad
		\tau_{\mathrm{skew}}(h)=0
		\ \text{and}\ L(\rho_{n\gamma})=n\norm{\gamma}_{H^1/\Gamma_\omega}
		\ \text{for all }n.
		\]
	\end{corollary}
	
	\section{Consequences of our study} \label{Sec-4}
	We now record several consequences, drawing on analogous results in
	Hamiltonian dynamics.
	
	\begin{definition}\label{symplecticdisplacement}
		The symplectic displacement energy \(e_s(U)\) of a non-empty subset
		\(U \subset M\) is
		\[
		e_s(U) = \inf\{ \norm{\phi}_{HL}^{(1, \infty)} \mid \phi \in G_\omega(M),\ \phi(U)\cap U = \emptyset\}
		\]
		if some element of \(G_\omega(M)\) displaces \(U\), and
		\(e_s(U) = +\infty\) otherwise.
	\end{definition}
	If \(U\subset V\), then \(e_s(U) \leq e_s(V)\).
	
	\begin{corollary}[Non-degeneracy of the Hofer-like norm]
		\label{cor:HoferNorm}
		Let \((M,\omega)\) be a closed symplectic manifold. If
		\(\phi\in G_\omega(M)\) satisfies \(\norm{\phi}_{HL}=0\),
		then \(\phi\in\Ham(M,\omega)\) and \(\phi=\id\).
	\end{corollary}
	
	\begin{proof}
		Assume \(\phi\in G_\omega(M)\) with \(\norm{\phi}_{HL}=0\). For every
		\(\epsilon>0\), choose a symplectic isotopy
		\(\Xi_\epsilon\in\Iso(\phi)_\omega\) from \(\id\) to \(\phi\) such
		that \(l_{HL}(\Xi_\epsilon)<\epsilon\). Write the harmonic part of its
		generator as \(h_t^\epsilon\in\Harm^1(M,g)\) and set
		\[
		\alpha_\epsilon:=\int_0^1 h_t^\epsilon\,dt\in H^1_{dR}(M,\R).
		\]
		Since \(\Xi_\epsilon\) connects \(\id\) to \(\phi\) and the pullback
		\(\phi_t^*\) acts trivially on \(H^1_{dR}(M,\R)\), the class
		\(\alpha_\epsilon\) is a lift of \(\Flux(\phi)\) under the quotient
		map \(\pi'\colon H^1_{dR}(M,\R)\to H^1_{dR}(M,\R)/\Gamma_\omega\), and
		by Minkowski's integral inequality,
		\begin{equation}\label{eq:alpha-eps-small}
			\norm{\alpha_\epsilon}_{L^2}
			\le
			\int_0^1\norm{h_t^\epsilon}_{L^2}\,dt
			\le
			l_{HL}(\Xi_\epsilon)
			<
			\epsilon .
		\end{equation}
		
		Since \(\Gamma_\omega\) is discrete, there exists
		\(\delta=\delta(M,\omega,g)>0\) such that
		\[
		\gamma\in\Gamma_\omega\ \text{and}\
		\norm{\gamma}_{L^2}<\delta
		\implies
		\gamma=0 .
		\tag{\(\star\)}
		\]
		
		\smallskip
		\noindent\emph{The lifts are eventually constant.} Fix
		\(\epsilon_1,\epsilon_2\in(0,\delta/2)\). Both \(\alpha_{\epsilon_1}\)
		and \(\alpha_{\epsilon_2}\) are lifts of \(\Flux(\phi)\) under
		\(\pi'\), hence their difference lies in the kernel of \(\pi'\):
		$
		\alpha_{\epsilon_1}-\alpha_{\epsilon_2}\in\ker\pi'=\Gamma_\omega .
		$
		By \eqref{eq:alpha-eps-small} and the triangle inequality,
		\[
		\norm{\alpha_{\epsilon_1}-\alpha_{\epsilon_2}}_{L^2}
		\le
		\norm{\alpha_{\epsilon_1}}_{L^2}
		+
		\norm{\alpha_{\epsilon_2}}_{L^2}
		<
		\frac{\delta}{2}+\frac{\delta}{2}
		=
		\delta .
		\]
		Applying \((\star)\) with
		\(\gamma=\alpha_{\epsilon_1}-\alpha_{\epsilon_2}\in\Gamma_\omega\)
		gives \(\alpha_{\epsilon_1}=\alpha_{\epsilon_2}\). Hence the family
		\(\{\alpha_\epsilon\}_{\epsilon\in(0,\delta/2)}\) is constant; write
		\(\alpha\) for its common value.
		
		\smallskip
		\noindent\emph{The constant is zero.} By
		\eqref{eq:alpha-eps-small} and the definition of \(\alpha\),
		$
		\norm{\alpha}_{L^2}
		=
		\norm{\alpha_\epsilon}_{L^2}
		<
		\epsilon$ 
		for every $\epsilon\in(0,\delta/2).$ 
		Letting \(\epsilon\to 0\) gives \(\norm{\alpha}_{L^2}=0\), hence
		\(\alpha=0\in H^1_{dR}(M,\R)\).
		
		\smallskip
		\noindent\emph{Consequences.} Since \(\alpha_\epsilon=\alpha=0\) is a
		lift of \(\Flux(\phi)\), we get \(\Flux(\phi)=0\) in
		\(H^1_{dR}(M,\R)/\Gamma_\omega\), so \(\phi\in\Ham(M,\omega)\).
		Moreover, for small \(\epsilon\) the isotopy \(\Xi_\epsilon\) has
		flux lift \(\widetilde{\Flux}(\Xi_\epsilon)=[\alpha_\epsilon]=0\), so
		\(\Xi_\epsilon\) is a zero-flux path from \(\id\) to \(\phi\), i.e.\
		\(\Xi_\epsilon\in\ZeroFlux{\phi}\). 
		Applying Lemma~\ref{lem:mainineq} with \(\Phi=\Xi_\epsilon\) and the
		constant loop \(\Psi\), and using that \(\Xi_\epsilon\) and \(\Psi\)
		have the same (zero) flux, we obtain
		\[
		\norm{\phi}_H
		\le
		\Bigl(1+3L_0\sup_{(s,t,x)\in[0,1]^2\times M}
		\norm{V^{\Xi_\epsilon}_{s,t}(x)}_g\Bigr)\,l_{HL}(\Xi_\epsilon).
		\]
		By Lemma~\ref{lem:diameter_bound}, there exist constants
		\(C_0,C_1>0\) depending only on \((M,\omega,g)\) such that
		\[
		\sup_{(s,t,x)\in[0,1]^2\times M}
		\norm{V^{\Xi_\epsilon}_{s,t}(x)}_g
		\le
		C_0\bigl(e^{C_1\epsilon}-1\bigr).
		\]
		Therefore
		\[
		\norm{\phi}_H
		\le
		\Bigl(1+3L_0\,C_0\bigl(e^{C_1\epsilon}-1\bigr)\Bigr)\,\epsilon
		\;\xrightarrow[\epsilon\to 0]{}\;0 .
		\]
		Hence \(\norm{\phi}_H=0\). By non-degeneracy of the Hofer norm on
		\(\Ham(M,\omega)\), \(\phi=\id\).
	\end{proof}

\begin{corollary}[Stability of \(C^0\)-limits]
	\label{cor:C0-Limits}
	Let \((M,\omega)\) be a closed symplectic manifold. Let
	\(\{\phi_i^t\}\) be a sequence of symplectic isotopies, \(\{\psi^t\}\)
	another symplectic isotopy, and \(\phi\colon M\to M\) a map such that
	\((\phi_i^1)\) converges uniformly to \(\phi\) and
	\(l_{HL}^{(1,\infty)}(\{\psi^t\}^{-1}\circ\{\phi_i^t\})\to 0\)
	as \(i\to\infty\). Then \(\phi = \psi^1\).
\end{corollary}

\begin{proof}
	For each \(i\), define the symplectic isotopy
	$
	\Xi_i:=\{\psi^t\}^{-1}\circ\{\phi_i^t\},
	$
	which starts at \(\id\) and terminates at
	$
	\eta_i:=(\psi^1)^{-1}\circ\phi_i^1 .
	$
	By hypothesis, \(l_{HL}(\Xi_i)\to 0\), and
	\(\eta_i\to\eta:=(\psi^1)^{-1}\circ\phi\) uniformly on \(M\).
	
	\smallskip
	\noindent\emph{Step 1: Banyaga's decomposition of \(\Xi_i\).} By
	Banyaga's decomposition \cite{Ban10}, the isotopy \(\Xi_i\) factors
	as
	$
	\Xi_i=\rho_i\circ\mu_i,
	$
	where \(\rho_i\) is a harmonic isotopy and \(\mu_i\) is a
	Hamiltonian isotopy, and the Hofer-like length is additive:
	$
	l_{HL}(\Xi_i)=l_{HL}(\rho_i)+l_H(\mu_i).
	$
	Since \(l_{HL}(\Xi_i)\to 0\), both \(l_{HL}(\rho_i)\to 0\) and
	\(l_H(\mu_i)\to 0\). Write \(\rho_i^1\), \(\mu_i^1\) for the
	time-one maps, so that
	$
	\eta_i=\rho_i^1\circ\mu_i^1 .
	$
	
	\smallskip
	\noindent\emph{Step 2: the harmonic factor is \(C^0\)-small.} Let
	\(\Harm_t^{\rho_i}:=\iota_{X_t^{\rho_i}}\omega\) be the harmonic
	generator of \(\rho_i\), so that
	$
	l_{HL}(\rho_i)=\int_0^1\norm{X_t^{\rho_i}}_{L^2}\,dt
	\;\xrightarrow[i\to\infty]{}\;0 .
	$
	Because \(\Harm^1(M,g)\) is finite-dimensional, there exists a
	constant \(C_g>0\) depending only on \((M,g)\) such that
	\(\norm{\alpha}_{C^0}\le C_g\norm{\alpha}_{L^2}\) for every
	harmonic \(1\)-form \(\alpha\). Therefore
	\[
	\int_0^1\norm{X_t^{\rho_i}}_{C^0}\,dt
	\;\le\;
	C_g\,l_{HL}(\rho_i)
	\;\xrightarrow[i\to\infty]{}\;0 .
	\]
	Setting \(\epsilon_i:=C_g\,l_{HL}(\rho_i)\to 0\), we obtain
	\[
	d\bigl(\rho_i^1(x),\,x\bigr)
	\;\le\;
	\int_0^1\norm{X_t^{\rho_i}}_{C^0}\,dt
	\;\le\;\epsilon_i
	\qquad\text{for every }x\in M .
	\]
	So \(\rho_i^1\to\id\) uniformly on \(M\) as \(i\to\infty\).
	
	\smallskip
	\noindent\emph{Step 3: the Hamiltonian factor converges \(C^0\).}
	Since \(\eta_i=\rho_i^1\circ\mu_i^1\), we have
	$
	\mu_i^1=(\rho_i^1)^{-1}\circ\eta_i .
	$
	By hypothesis \(\eta_i\to\eta\) uniformly, and by Step~2
	\((\rho_i^1)^{-1}\to\id\) uniformly (since \(\rho_i^1\to\id\)).
	Therefore
	$
	\mu_i^1\;\xrightarrow[i\to\infty]{}\;\eta$ uniformly on $M.$
	
	\smallskip
	\noindent\emph{Step 4: \(\eta=\id\).} Since
	\(\mu_i^1\in\Ham(M,\omega)\) and \(l_H(\mu_i)\to 0\), we have
	\[
	\norm{\mu_i^1}_H
	\;\le\;l_H(\mu_i)
	\;\xrightarrow[i\to\infty]{}\;0 .
	\]
	By the \(C^0\)-continuity of the Hofer norm at the identity on
	\(\Ham(M,\omega)\) (see Oh \cite{Oh2005} and
	Humili\`ere-Leclercq-Seyfaddini \cite{HLS2015}, see also
	M\"uller \cite{Muller2017}), a sequence of Hamiltonian
	diffeomorphisms with vanishing Hofer norm converges to the identity
	uniformly on \(M\): there exists \(\eta_0\colon[0,\infty)\to[0,\infty)\)
	with \(\eta_0(t)\to 0\) as \(t\to 0\) such that
	$
	d_{C^0}(\mu_i^1,\id)\;\le\;\eta_0\bigl(\norm{\mu_i^1}_H\bigr)
	\;\xrightarrow[i\to\infty]{}\;0 .
	$
	Hence, \(\mu_i^1\to\id\) uniformly. Combined with Step~3,
	\(\eta=\id\), i.e.
	$
	(\psi^1)^{-1}\circ\phi=\id$ is equivalent to 
	$
	\phi=\psi^1.$ 
\end{proof}

\begin{corollary}[Positivity of the displacement energy]
	\label{cor:positivity}
	Let \((M,\omega,g)\) be a closed symplectic manifold, and let
	\(U\subset M\) be a non-empty subset with non-empty interior that can
	be displaced by an element of \(G_\omega(M)\). Then
	\[
	e_s(U)
	:=
	\inf\{\norm{\phi}_{HL}:\phi\in G_\omega(M),\
	\phi(U)\cap U=\emptyset\}
	>0 .
	\]
\end{corollary}

\begin{proof}
	Suppose for contradiction that \(e_s(U)=0\). Then there exists a
	sequence \(\{\phi_i\}\subset G_\omega(M)\) with
	\(\phi_i(U)\cap U=\emptyset\) and \(\norm{\phi_i}_{HL}\to 0\).
	
	\smallskip
	\noindent\emph{Step 1: reduction to a ball.} Since \(U\) has non-empty
	interior, there exist \(x_0\in M\) and \(r>0\) such that the closed
	ball \(U':=\overline{B_{r/2}(x_0)}\) is contained in the interior of
	\(U\). Since \(\phi_i(U')\cap U'\subseteq\phi_i(U)\cap U=\emptyset\),
	each \(\phi_i\) displaces \(U'\). The configuration of the sets \(U\),
	\(U'\), \(B_r(x_0)\) and the displaced copy \(\phi_i(U')\) is
	illustrated in Figure~\ref{fig:setup}.
	
	\begin{figure}[ht]
		\centering
		\begin{tikzpicture}[scale=1.0, >=Stealth, line join=round]
			\draw[thick, rounded corners=15pt] (-4.2,-2) rectangle (4.2,2);
			\node at (3.7,1.6) {$M$};
			
			\draw[thick] (0.7,0) circle (1.4);
			\node at (0.7,1.7) {$U$};
			
			\draw[dashed, gray] (0.7,0) circle (0.9);
			\node[gray] at (1.7,0.9) {\small $B_r(x_0)$};
			
			\fill[gray!30] (0.7,0) circle (0.45);
			\draw[thick] (0.7,0) circle (0.45);
			\node at (0.7,-0.75) {$U'$};
			
			\fill (0.7,0) circle (0.05);
			\node[below left] at (0.7,0) {$x_0$};
			
			\fill[gray!30] (-2.5,0) circle (0.55);
			\draw[thick] (-2.5,0) circle (0.55);
			\node at (-2.5,-0.85) {$\phi_i(U')$};
			
			\draw[->, thick] (0.05,0.55)
			.. controls (-0.8,1.5) and (-1.8,1.5) .. (-2.05,0.5);
			\node at (-1.2,1.55) {\small $\phi_i$};
			
			\node at (0.7,-1.6) {\small $\phi_i(U')\cap U=\emptyset$};
		\end{tikzpicture}
		\caption{Setup of the displacement argument. The closed ball
			\(U'=\overline{B_{r/2}(x_0)}\) is contained in the open set \(U\). The
			image \(\phi_i(U')\) lies in \(M\setminus U\), since \(\phi_i\)
			displaces \(U\). The dashed circle \(B_r(x_0)\) contains \(U'\) with
			margin \(r/2\).}
		\label{fig:setup}
	\end{figure}
	
	\smallskip
	\noindent\emph{Step 2: Banyaga's decomposition.} For each \(i\), choose
	a symplectic isotopy \(\Xi_i\) from \(\id\) to \(\phi_i\) with
	\(l_{HL}(\Xi_i)\le\norm{\phi_i}_{HL}+1/i\). By Banyaga's decomposition
	\cite{Ban10}, \(\Xi_i=\rho_i^t\circ\mu_i^t\), where \(\rho_i^t\) is a
	harmonic isotopy and \(\mu_i^t\) a Hamiltonian isotopy, and
	$
	l_{HL}(\Xi_i)=l_{HL}(\rho_i^t)+l_H(\mu_i^t).
	$
	Hence \(l_{HL}(\rho_i^t)\to 0\) and \(l_H(\mu_i^t)\to 0\). Write
	\(\rho_i:=\rho_i^1\in\Harm\), \(\mu_i:=\mu_i^1\in\Ham\), so that
	\(\phi_i=\rho_i\circ\mu_i\) and
	$
	\norm{\mu_i}_H\;\le\;l_H(\mu_i^t)\;\xrightarrow[i\to\infty]{}\;0.
	$
	
	\smallskip
	\noindent\emph{Step 3: the harmonic factor is \(C^0\)-small.} Let
	\(\Harm_t^i:=\iota_{X_t^i}\omega\) be the harmonic generator of
	\(\rho_i^t\), so that
	\[
	l_{HL}(\rho_i^t)=\int_0^1\norm{X_t^i}_{L^2}\,dt\;\xrightarrow[i\to\infty]{}\;0 .
	\]
	Because \(\Harm^1(M,g)\) is finite-dimensional, there exists a constant
	\(C_g>0\) depending only on \((M,g)\) such that
	\(\norm{\alpha}_{C^0}\le C_g\norm{\alpha}_{L^2}\) for every harmonic
	\(1\)-form \(\alpha\). Therefore
	\[
	\int_0^1\norm{X_t^i}_{C^0}\,dt
	\;\le\;
	C_g\int_0^1\norm{X_t^i}_{L^2}\,dt
	\;=\;
	C_g\,l_{HL}(\rho_i^t)
	\;\xrightarrow[i\to\infty]{}\;0 .
	\]
	Setting \(\epsilon_i:=C_g\,l_{HL}(\rho_i^t)\to 0\), and integrating
	the ODE for \(\rho_i^t\) along the path \(t\mapsto\rho_i^t(x)\), we
	obtain
	\[
	d\bigl(\rho_i(x),\,x\bigr)
	\;\le\;
	\int_0^1\norm{X_t^i}_{C^0}\,dt
	\;\le\;\epsilon_i
	\qquad\text{for every }x\in M .
	\]
	So \(\rho_i\) is \(\epsilon_i\)-close to the identity in the
	\(C^0\)-topology, with \(\epsilon_i\to 0\).
	
	\smallskip
	\noindent\emph{Step 4: the Hamiltonian factor displaces \(U'\).} Fix
	\(i\) large enough that \(\epsilon_i<r/2\). We claim
	\(\mu_i(U')\cap U'=\emptyset\). Suppose otherwise: then there exists
	\(x\in U'\) with \(\mu_i(x)\in U'\), applying \(\rho_i\) we get 
	\[
	\phi_i(x)=\rho_i\bigl(\mu_i(x)\bigr).
	\]
	Since \(\rho_i\) is \(\epsilon_i\)-close to the identity,
	\(d(\phi_i(x),\mu_i(x))\le\epsilon_i\), and since
	\(\mu_i(x)\in U'=\overline{B_{r/2}(x_0)}\),
	\[
	d\bigl(\phi_i(x),x_0\bigr)
	\;\le\;
	d\bigl(\phi_i(x),\mu_i(x)\bigr)+d\bigl(\mu_i(x),x_0\bigr)
	\;\le\;
	\epsilon_i+\tfrac{r}{2}
	\;<\;r .
	\]
	Hence \(\phi_i(x)\) lies in the open ball \(B_r(x_0)\subseteq U\). But
	\(\phi_i(x)\in\phi_i(U')\subseteq\phi_i(U)\subseteq M\setminus U\): 
	contradiction. Therefore, \(\mu_i(U')\cap U'=\emptyset\).\\
	
	\begin{figure}[ht]
	\centering
	\begin{tikzpicture}[scale=2.4, >=Stealth, line join=round]
		\draw[dashed, gray, thick] (0,0) circle (1);
		\node[gray] at (1.25,0.78) {\small $B_r(x_0)$};
		
		\fill[gray!30] (0,0) circle (0.5);
		\draw[thick] (0,0) circle (0.5);
		\node at (-0.80,-0.35) {\small $U'$};
		
		\fill (0,0) circle (0.02);
		\node[below left] at (-0.04,-0.03) {\small $x_0$};
		
		\fill (-0.30,0.10) circle (0.028);
		\node[above left] at (-0.30,0.12) {\small $x$};
		
		\fill (0.10,-0.20) circle (0.028);
		\node[below] at (0.10,-0.26) {\small $\mu_i(x)$};
		
		\fill (0.32,-0.06) circle (0.028);
		\node[right] at (0.35,-0.06) {\small $\phi_i(x)$};
		
		\draw[<->, thick, red] (0.10,-0.20) -- (0.32,-0.06);
		\node[red, below right] at (0.24,-0.15) {\tiny $\le\epsilon_i$};
		
		\draw[->, gray, thick]
		(0.16,-0.24) .. controls (0.28,-0.30) and (0.42,-0.20) .. (0.36,-0.06);
		\node[gray, below] at (0.30,-0.34) {\tiny $\rho_i$};
		
		\node[blue, align=left] at (1.65,0.30)
		{\small Since $\epsilon_i<r/2$,\\
			\small $\phi_i(x)\in B_r(x_0)$};
		\draw[->, blue, thick] (1.30,0.28) -- (0.55,0.15);
		
		\node[blue, align=left] at (1.65,-0.55)
		{\small But $B_r(x_0)\subset U$, so\\
			\small $\phi_i(x)\in U$, contradiction};
		\draw[->, blue, thick] (1.30,-0.55) -- (0.35,-0.10);

	\end{tikzpicture}
	\caption{Metric contradiction in Step~4. Assuming
		\(\mu_i(x)\in U'\) for some \(x\in U'\), the point
		\(\phi_i(x)=\rho_i(\mu_i(x))\) lies within \(\epsilon_i<r/2\) of
		\(\mu_i(x)\in U'=\overline{B_{r/2}(x_0)}\). Hence
		\(\phi_i(x)\in B_r(x_0)\subseteq U\), contradicting
		\(\phi_i(U')\cap U=\emptyset\).}
	\label{fig:contradiction}
\end{figure}
	\smallskip
	\noindent\emph{Step 5: contradiction.} Since
	\(\mu_i\in\Ham(M,\omega)\) and \(\mu_i(U')\cap U'=\emptyset\), we get
	\[
	\norm{\mu_i}_H\;\ge\;e_H(U')>0,
	\]
	where \(e_H(U')\) is the classical Hofer displacement energy of the
	ball \(U'\), which is strictly positive by the energy-capacity
	inequality \cite{Hofer,LMP92,Polt01,Ban18}. But by Step 2,
	\(\norm{\mu_i}_H\le l_H(\mu_i^t)\to 0\): contradiction. This completes the proof. 
\end{proof}

\section*{Conclusion}

Buss and Leclercq \cite[Proposition~2.6]{BusLec11} proved that their
family of pseudo-distances \(d_{(\mu,c)}\) on
\(\mathrm{Symp}(M,\omega)\) has Hamiltonian degeneracy when
\(c\neq 0\): namely, \(d_{(\mu,c)}(\id,\phi)=0\) if and only if
\(\phi\in\Ham(M,\omega)\). Their argument also shows that the
restriction of \(d_{(\mu,c)}\) to \(\Ham(M,\omega)\) induces the same
topology as the Hofer metric. The present paper strengthens this in two directions. First, the
semidirect-product formula of Theorem~\ref{thm:A} expresses the
Hofer-like norm of a general element of \(G_\omega(M)\) as an infimum
over the harmonic locus plus a Hamiltonian residue; specialising to
zero flux gives the geometric criterion
\(L(\rho)\ge\norm{\rho}_H\), \(\rho\in\Harm(0)\), under which the
inclusion of the Hamiltonian group into the identity component is an
isometric embedding (Theorem~\ref{thm:C}). Second, under the explicit
boundedness hypothesis (H\(_K\)) on the two-parameter vector-field
family attached to zero-flux paths, Theorem~\ref{thm:B} gives a
bi-Lipschitz bound with explicit distortion constant \(1+3L_0K\). The sharpness condition and hypothesis (H\(_K\)) are logically
independent, and the paper keeps them separate.

\smallskip
\noindent\emph{Sharpness.} The condition
\(L(\rho)\ge\norm{\rho}_H\) on \(\Harm(0)\) is equivalent, by
Theorem~\ref{thm:C}, to the isometry of the inclusion
\[
\bigl(\Ham(M,\omega),\norm{\cdot}_H\bigr)
\hookrightarrow
\bigl(G_\omega(M),\norm{\cdot}_{HL}\bigr).
\]
When it fails, the inclusion is not isometric; the two norms may
nevertheless be bi-Lipschitz equivalent with some constant \(C>1\).

\smallskip
\noindent\emph{Bi-Lipschitz property.} This is controlled separately
by hypothesis (H\(_K\)). Theorem~\ref{thm:B} shows that (H\(_K\))
implies bi-Lipschitz equivalence with \(C=1+3L_0K\). On every closed
symplectic manifold, the \emph{local} version of (H\(_K\)) holds on
each Hofer ball
\[
B_R:=\{\phi\in\Ham(M,\omega):\norm{\phi}_H\le R\},
\]
with \(K(R)=\exp(2C_{\mathrm{geom}}R)-1\), by
Lemma~\ref{lem:diameter_bound} and the admissibility condition
\(l_{HL}(\Xi)\le 2\norm{\phi}_H\). Consequently, on every closed
symplectic manifold the inclusion is bi-Lipschitz on each Hofer ball
\(B_R\), with distortion at most
\(1+3L_0K(R)\), a bound that grows exponentially with \(R\) but is
finite for every fixed \(R\). Whether (H\(_K\)) holds \emph{globally}, i.e.\ whether \(K\) can be
chosen independent of the Hofer radius, is a separate question from
the sharpness condition. It reduces to whether
\[
\sup\Bigl\{
\sup_{(s,t,x)}\norm{V^\Xi_{s,t}(x)}_g
\Bigr\}
\;<\;
+\infty,
\]
over all \(\phi\in\Ham(M,\omega)\) and all
\(\Xi\in\ZeroFlux{\phi}\) with
\(l_{HL}(\Xi)\le 2\norm{\phi}_H\). Neither condition implies the
other: the sharpness condition is stronger than bi-Lipschitz
equivalence (it gives isometry, i.e.\ \(C=1\)), whereas (H\(_K\))
is a statement about the two-parameter family \(V^\Xi_{s,t}\), which seems 
 independent of the harmonic locus \(\Harm(0)\).

\smallskip
The two open questions can be phrased concretely on the test cases
introduced in the paper. On \((T^2,\omega,g_\varepsilon)\)
(Example~\ref{ex:non-vacuous}), the sharpness condition is verified
on the infinite family \(\{\rho_n\}_{n\in\mathbb Z}\) but is not
known to hold on all of \(\Harm(0)\). On products
\(N\times T^{2k}\) with \(N\) simply connected
(Proposition~\ref{prop:singleton-locus}), the harmonic locus is a
singleton in every flux class, and the semidirect-product formula
becomes the exact identity
\[
\norm{\phi}_{HL}=L(\rho_\gamma)+\norm{\phi\rho_\gamma^{-1}}_H .
\]

The geometric content of the constant \(1+3L_0K\) is the following.
It measures the maximal distortion between the Hofer norm and the
Hofer-like energy on the Hamiltonian subgroup, uniformly over all
Hamiltonian diffeomorphisms. It equals \(1\) precisely when the
harmonic directions do not shorten Hofer distances, i.e.\ when
\(L(\rho)\ge\norm{\rho}_H\) for every \(\rho\in\Harm(0)\).
Whether the constant can be taken equal to \(1\) on a given manifold
is equivalent, by Theorem~\ref{thm:C}, to the sharpness condition.
The paper provides two verifications of this condition: the
Ricci non-negative case of Proposition~\ref{prop:ricci-commuting},
covering flat tori and Calabi--Yau manifolds, and the non-vacuous
family on \((T^2,\omega,g_\varepsilon)\) of
Example~\ref{ex:non-vacuous}. Whether the sharpness condition holds
unconditionally on every closed symplectic manifold, or
equivalently whether there exist \((M,\omega,g)\) with \(b_1(M)>0\)
and \(\rho\in\Harm(0)\setminus\{\id\}\) satisfying
\(L(\rho)<\norm{\rho}_H\), remains open.
	
	\section*{Acknowledgments}
	This work is dedicated to Professor Augustin Banyaga, in heartfelt
	appreciation for his unwavering commitment to mathematics and for
	introducing me to the fascinating world of symplectic topology.
	
	I extend my deepest gratitude to the CNRS, DAAD, CIRM, and IMJ-Paris for
	their invaluable support and contributions, which have been instrumental
	in my academic journey.
	
	\section*{Declarations}
	\subsection*{Ethics approval}
	Not applicable.
	\subsection*{Competing interests}
	The author declares that there are no competing interests.
	\subsection*{Funding}
	The author received no financial support for the research, authorship,
	and/or publication of this article.

\end{document}